\documentclass[11pt,reqno]{amsart}
\usepackage{geometry}
\usepackage{xcolor}
\usepackage{amsthm}
\usepackage{amssymb}
\usepackage{hyperref}
\usepackage{arydshln}
\usepackage{mathrsfs}
 \usepackage{booktabs} 

\usepackage{bm}
\usepackage{bbm}
\usepackage{enumitem}  % Customisable lists
\usepackage{empheq}    % Boxed equations
\usepackage{microtype} % Better typography

\usepackage[utf8]{inputenc} % usually not needed (loaded by default)
\usepackage[T1]{fontenc}

\makeatletter
\numberwithin{equation}{section}
\numberwithin{figure}{section}

\newcommand{\Q}{\mathbb{Q}}
\newcommand{\C}{\mathbb{C}}
\newcommand{\R}{\mathbb{R}}

\newcommand{\F}{\mathcal{F}}

\renewcommand{\P}{\mathbb{P}}
\newcommand{\E}{\mathbb{E}}

\newtheorem{Theorem}{Theorem}[section]
\newtheorem{Proposition}[Theorem]{Proposition}\newtheorem{Corollary}[Theorem]{Corollary}\newtheorem{Lemma}[Theorem]{Lemma}\newtheorem{Remark}[Theorem]{Remark}\newtheorem{Definition}[Theorem]{Definition}

\numberwithin{equation}{section}

\makeatother

\begin{document}

\title[Boundary behaviour of the Volterra square-root process]{Boundary behaviour of the Volterra square-root process}

\author{Martin Friesen}
\author{Stefan Gerhold}
\author{Kristof Wiedermann}
\address[Martin Friesen]
{School of Mathematical Sciences
\\ Dublin City University
\\ Glasnevin, Dublin 9, Ireland}
\email{martin.friesen@dcu.ie}

\address[Stefan Gerhold]
{Financial and Actuarial Mathematics, TU Wien, Vienna, Austria}
\email{sgerhold@fam.tuwien.ac.at}

\address[Kristof Wiedermann]
{Financial and Actuarial Mathematics, TU Wien, Vienna, Austria}
\email{kristof.wiedermann@tuwien.ac.at}

\date{\today}

% 60H20  Stochastic integral equations
% 45D05  Volterra integral equations
% 60K50  Anomalous diffusion models 


\begin{abstract}
In this work, we study the boundary behaviour of the Volterra square-root process on $\R_+$. For regular Volterra kernels, we establish a time-dependent Feller condition that guarantees that the process does not hit zero on $[0,T]$, and prove finiteness of negative $p$-moments. For rough kernels that are regularly varying at zero, we show that the process necessarily hits zero with positive probability, and that its law has an atom at the boundary. Finally, we establish analogous results for the limit distribution. Our proofs are based on comparison principles for Volterra integral equations and generalized Riemann--Liouville fractional equations. The latter provide us with upper and lower bounds for the solution of the associated Volterra Riccati equation, and hence bounds on the Laplace transform. As an application, we study the structure of equivalent martingale measures in the Volterra Heston model. For the rough case, we show that equivalent martingale measures exist only under very restrictive assumptions on the drift under the real-world measure.
\end{abstract}

\maketitle

\allowdisplaybreaks

\section{Introduction}

\subsection{Overview}

Rough volatility models have emerged as a powerful framework in mathematical finance, motivated by the empirical observation that volatility is significantly rougher than predicted by models based on standard Brownian motion \cite{MR3805308}. By incorporating drivers with fractional memory, typically corresponding to a Hurst parameter $H \in (0,1/2)$, these models capture both the fine-scale behaviour of variance dynamics and fit the implied volatility surface, in particular the pronounced short-maturity skew, using only a few parameters. Beyond roughness, path dependence has also been argued to play an important role \cite{MR4960404, MR4625912, romer2022empirical}. From a mathematical perspective, roughness rules out the semimartingale property, while path dependence leads to non-Markovian processes as demonstrated in \cite{FGW25}.

Affine Volterra processes provide a flexible class of models in which roughness and path dependence can be naturally incorporated via convolution with a Volterra kernel, while analytical tractability is retained through the affine transform formula. In this work, we focus on the Volterra square-root diffusion process on $\R_+$, also known as the Volterra Cox-Ingersoll-Ross process, that models the variance process in the Volterra Heston model and constitutes the most prominent example of an affine Volterra process. Its dynamics is given by the stochastic Volterra equation 
\begin{align}\label{eq: VCIR}
    X_t = x_0 + \int_0^t K(t-s)(b + \beta X_s)\,\mathrm{d}s
    + \sigma \int_0^t K(t-s)\sqrt{X_s}\,\mathrm{d}B_s,
\end{align}
where $B$ is a standard Brownian motion, $b \geq 0$, $\beta \in \mathbb{R}$ governs the mean-reversion structure, and, when $\beta<0$, the quantity $-b/\beta$ is the long-term mean. Throughout this work, we assume $\sigma > 0$, and that the Volterra kernel $K$ satisfies the following condition:

\begin{enumerate}
    \item[(K)] $0 \not\equiv K \in L_{\mathrm{loc}}^2(\mathbb{R}_+)$ is completely monotone, and there exist $\gamma \in (0,1]$ and, for each $T>0$, a constant $C_T>0$ such that
    \begin{align}\label{eq: K1}
        \int_0^h K(t)^2\, \mathrm{d}t + \int_0^T |K(t+h)-K(t)|^2\, \mathrm{d}t \leq C_T h^{2\gamma},
        \qquad h > 0.
    \end{align}
\end{enumerate}
Under assumption (K), for any $b,x_0 \in \mathbb{R}_+$, $\sigma > 0$, and $\beta \in \mathbb{R}$, equation \eqref{eq: VCIR} admits a unique nonnegative weak solution with Hölder continuous sample paths of any order $\eta \in (0,\gamma)$ and finite moments of all orders. For additional details on the construction of this process we refer to \cite[Theorem 6.1]{MR4019885}.

Equation \eqref{eq: VCIR} covers the classical CIR process when $K(t) \equiv 1$. The fractional kernel $K(t) = \frac{t^{\alpha-1}}{\Gamma(\alpha)}$ with $\alpha = H + \frac{1}{2} \in (\frac{1}{2},1)$ corresponds to the rough CIR process, which constitutes the key component of the rough Heston model \cite{MR3778355,MR3905737}. Other kernels also appear in the literature. For instance, $K(t) = \sum_{j=1}^N c_j e^{-\lambda_j t}$ corresponds to an $N$-factor Markovian approximation \cite{MR3934104}, hyperbolic kernels $K(t) = c_H (\eta+t)^{H - \frac{1}{2}}$ with $H < \frac{1}{2}$ and $\eta > 0$ model hyper-rough volatility,  while $K(t) = \log(1+t^{-\alpha})$ provides a kernel that decays like $t^{-\alpha}$ for large times with small-time fluctuations governed on a logarithmic scale. 

In this work, we study the boundary behaviour of the Volterra square-root process, i.e., the hitting time of zero and finiteness of negative moments. For Markovian models, e.g. the CIR process with $K \equiv 1$, or more generally (multi-type CBI processes), such results are studied in terms of the Riccati equation and asymptotic behaviour of the branching and immigration mechanisms \cite{MR3264444, MR3263091, MR408016}. Extensions to multi-type CBI processes have been studied in~\cite{MR4195178}, and in~\cite{MR3849816} for processes with a countable number of types. As already established in \cite{MR408016}, the probability of an atom in zero is closely related to the maximal solution of the corresponding Riccati equation. To prove analogous results for the Volterra square-root process, we establish a similar relation for the maximal solution of a Volterra Riccati equation, and study its asymptotics.

At the same time as this work, a related result appeared in \cite{BJ26}. Besides optimal conditions for the finiteness of moments in the rough Bergomi model, the authors prove that the variance process in the rough Heston model (the rough Cox-Ingersoll-Ross process) has an atom in zero, and derive sufficient Feller boundary conditions for regular Volterra kernels in the Volterra square-root process. In the regular case, we obtain two-sided bounds and asymptotic optimality of the Feller condition, whereas \cite{BJ26} allows for a time-dependent $x_0$. For the rough case, the results from \cite{BJ26} apply to the fractional kernel $K(t)= t^{\alpha-1}/\Gamma(\alpha)$ and coincide with our lower bound \eqref{eq:glob l bd}, whereas our framework covers a general class of Volterra kernels that are regularly varying at zero, and provides upper bounds. Finally, we also discuss the boundary behaviour of the limit distribution.

\subsection{The Volterra Riccati equation}

The Volterra square-root process is a particular case of affine Volterra processes on the state space $\R_+$, see \cite[Section 6]{MR4019885}. Hence, its characteristic function can be expressed via an affine transformation formula. Below we state such a formula for the Fourier-Laplace transform of $\int_{[0,t]} X_{t-s}\,\mu(\mathrm{d}s)$, where $\mu$ is a locally finite measure with values in $\C_- := \{u \in \C : \mathrm{Re}(u) \leq 0\}$. Let
\[
    R(u) = \beta u + \frac{\sigma^2}{2}u^2.
\]
By Lemma \ref{lemma: fractional Sobolev}, we may apply the results of \cite[Section 3]{FJ22}, and hence there exists a unique solution $\psi(\cdot;\mu) \in L_{\mathrm{loc}}^2(\R_+; \C_-)$ of
\begin{align}\label{eq: Volterra Riccati measure}
    \psi(t;\mu) = \int_{[0,t]} K(t-s)\,\mu(\mathrm{d}s) + \int_0^t K(t-s)R(\psi(s;\mu))\, \mathrm{d}s.
\end{align}
Moreover, the process $X$ satisfies the affine transformation formula
\begin{align}\label{eq: affine trafo}
    \E\left[ \exp\left(\int_{[0,t]} X_{t-s}\,\mu(\mathrm{d}s)\right) \right]
    = \exp\left( x_0 \mu([0,t]) + \int_0^t R(\psi(s;\mu))\, \mathrm{d}s + b \int_0^t \psi(s;\mu)\, \mathrm{d}s \right).
\end{align}
The case $\mu(\mathrm{d}s) = u \delta_0(\mathrm{d}s) + f(s)\,\mathrm{d}s$ was covered in \cite[Section 6]{MR4019885}, whereas the choice $\mu(\mathrm{d}s) = \sum_{j=1}^n u_j \delta_{t_n - t_j}(\mathrm{d}s)$ with $u_1,\dots,u_n \in \C_-$ and $0 < t_1 < \dots < t_n$ yields the Fourier-Laplace transform for the finite-dimensional distributions of $(X_{t_1}, \dots, X_{t_n})$.

The pointwise regularity of $\psi$ is more subtle and less explored. In \cite[Section 3]{FJ22}, it was shown that $\psi(\cdot;\mu) - (K \ast \mu)$ is continuous on $(0,\infty)$, and sufficient conditions for continuity at $t=0$ and local Hölder continuity were also established. For an equation towards one-dimensional CBI processes with jumps we refer to~\cite{MR4674624}. Our first main result complements this by showing that $\psi$ is actually smooth on the open interval $(0,\infty)$ whenever $(K \ast \mu)(t) = \int_{[0,t]} K(t-s)\,\mu(\mathrm{d}s)$ is smooth on $(0,\infty)$.

\begin{Theorem}\label{Theorem: regularity}
    Suppose that $K$ is regularly varying at $t=0$, that is,
    \begin{align}\label{eq: regularly varying}
        K(t) = \frac{t^{\alpha-1}}{\Gamma(\alpha)}\ell(t), \qquad t > 0,
    \end{align}
    where $\ell > 0$ is slowly varying at $0$ and $\alpha \in (0,1)$. Let $m \geq 0$ and $\eta \in (0,1)$. Then
    \[
        K \ast \mu \in C^{m,\eta}((0,\infty); \C_-)
        \ \Longrightarrow \
        \psi(\cdot;\mu) \in C^{m,\eta}((0,\infty); \C_-).
    \]
    In particular, if $K \ast \mu$ is $C^\infty$, then also $\psi(\cdot;\mu) \in C^{\infty}$.
\end{Theorem}
The proof of this theorem is given in Section 2. It is based on the smoothing property of the convolution operator combined with a bootstrapping argument. Let us remark that for smooth kernels $K$, direct differentiation of the convolutions $\int_{[0,t]}K(t-s)\, \mu(\mathrm{d}s), \int_0^t K(t-s)R(\psi(s;\mu))\, \mathrm{d}s$ shows that regularity in fact holds on the entire time domain $\R_+$. Thus, the treatment of weakly singular kernels constitutes the main contribution of Theorem~\ref{Theorem: regularity}, and subsequently allows us to verify the regularity assumptions in the comparison Theorem~\ref{thm:strict} from the appendix.

\subsection{Regular case}

A fundamental question in the study of affine processes on $\R_+$ concerns their boundary behaviour and the smoothness of their transition densities. For the classical CIR process (that is, when $K \equiv 1$), it is well known that $X$ does not hit the boundary if and only if the Feller boundary condition $\frac{2b}{\sigma^2} \geq 1$ is satisfied. Moreover, since the CIR process has a non-central $\chi$-squared distribution with $\frac{4b}{\sigma^2}$ degrees of freedom, one can verify that $\E[X_t^{-p}] < \infty$ is equivalent to $2b/\sigma^2 > p$. Below, we provide an extension of these results to the Volterra square-root process with $K \neq 1$.

Below we recall the definition of the resolvent of the first kind, which provides the key tool to capture the memory structure of the process. A locally finite Borel measure $L$ on $\R_+$ is called resolvent of the first kind associated with the Volterra kernel $K$, if it satisfies $L \ast K = K \ast L = 1$. Since $K$ is completely monotone and not identically zero, such a measure always exists \cite[Chapter 5, Theorem 5.4]{grippenberg}. With the convention $\infty^{-1} := 0$, it is given by
\begin{align}\label{eq: resolvent first kind}
    L(\mathrm{d}t) = K(0)^{-1}\delta_0(\mathrm{d}t) + L_0(t)\,\mathrm{d}t,
\end{align}
where $L_0 \in L_{\mathrm{loc}}^1(\R_+)$ is completely monotone. It follows from Lemma \ref{lemma: first kind asymptotics} that $L_0(0) = |K'(0)|/K(0)^2 < \infty$ whenever $K \in C^1(\R_+)$. To formulate our main result on the boundary behaviour for regular kernels $K \in C^1(\R_+)$, let us define $\lambda = \beta - L_0(0)$, and set
\[
        C(t) = \frac{2x_0}{\sigma^2 K(0)^2} \cdot \begin{cases}\frac{\lambda \mathrm{e}^{\lambda t}}{\mathrm{e}^{\lambda t} - 1},& \lambda \neq 0 \\ \frac{1}{t},&  \lambda = 0 \end{cases} \ \text{ and } \ \theta(t) = \frac{\sigma^2 K(0)^2}{2}\cdot \begin{cases} \frac{\mathrm{e}^{\lambda t} - 1}{\lambda},&  \lambda \neq 0 \\ t,&  \lambda = 0. \end{cases}
\]
The following is our main result on the boundary behaviour for regular Volterra kernels.

\begin{Theorem}[Boundary]\label{thm: boundary regular case}
    Suppose that $K \in C^1(\R_+)$, and let 
    \[
        \gamma_-(t) = \frac{2(x_0 L_0(t) + b)}{\sigma^2 K(0)} \ \text{ and } \ \gamma_+(t) = \frac{2\Lambda(t)}{\sigma^2 K(0)},
    \]
    with $\Lambda$ given by
    \[
        \Lambda(t) = x_0 L_0(0)\mathrm{e}^{\beta^+ K(0) t} + b \begin{cases} 1 + L_0(0) \frac{\mathrm{e}^{\beta K(0)t} - 1}{\beta},&  \beta \geq L_0(t)
        \\  \frac{L_0(0) - \beta}{L_0(t) - \beta},& \beta < L_0(t). \end{cases}
    \]
    Then $\P[X_t = 0] = 0$ for all $t > 0$. Moreover, let $p > 0$ and $t > 0$ satisfy $\gamma_-(t) > p$. If $x_0 = 0$, then 
   \[
     \frac{1}{\theta(t)^{p}} \frac{\Gamma(\gamma_{+}(t) - p)}{\Gamma(\gamma_{+}(t))} \leq \E[X_t^{-p}] \leq \frac{\Gamma(\gamma_{-}(t) - p)}{\Gamma(\gamma_{-}(t))} \frac{1}{\theta(t)^{p}},
    \]
    and if $x_0 > 0$, then 
    \[
        \frac{1}{\theta(t)^{p}} \mathrm{e}^{-C(t)} \frac{\Gamma(\gamma_+(t)-p)}{\Gamma(\gamma_+(t))} \leq \E[X_t^{-p}] \leq \frac{1}{\theta(t)^{p}} \left[ \left( \frac{C(t)}{2} \right)^{-p} + \mathrm{e}^{-C(t)/2} \frac{\Gamma(\gamma_-(t) - p)}{\Gamma(\gamma_-(t))} \right].
    \]
    If $p \geq \gamma_+(t)$, then $\E[X_t^{-p}] = \infty$ for any $x_0 \geq 0$.
\end{Theorem} 

The upper bounds provide finiteness of negative moments (and hence absence of an atom in zero), under the sufficient condition
\begin{align}\label{eq: Volterra Feller boundary condition}
    \gamma_-(t) > p,
\end{align}
which has the following interpretation. If $x_0 = 0$, then $\gamma_-(t) = \frac{2b}{\sigma^2 K(0)}$ and \eqref{eq: Volterra Feller boundary condition} reduces to the classical Feller condition under the scaling $b \longmapsto K(0)b$ and $\sigma \longmapsto K(0)\sigma$. Such scaling appears, e.g., in the semimartingale representation of the process (see, e.g., Section 4). However, when $x_0 > 0$, and $K \neq 1$ (hence $L_0 \neq 0$), \eqref{eq: Volterra Feller boundary condition} incorporates the memory effect of the process through the time-dependent effective drift $t \longmapsto x_0 L_0(t) + b$, which increases the number of finite negative moments in the case $x_0 > 0$. By complete monotonicity of $L_0$, we obtain  
\[
    \gamma_-(\infty) = \frac{2(x_0 L_0(\infty) + b)}{\sigma^2 K(0)} \leq \frac{2(x_0 L_0(0) + b)}{\sigma^2 K(0)} = \gamma_-(0),
\]
where according to Lemma \ref{lemma: first kind asymptotics} the evaluations $L_0(\infty)$ and $L_0(0)$ are given by
\begin{align}\label{eq: 7}
    L_0(0) = \frac{|K'(0)|}{K(0)^2} \qquad\text{and}\qquad
    L_0(\infty) = \left( \int_0^\infty K(t)\,\mathrm{d}t \right)^{-1}.
\end{align}
\begin{Remark}
    Assume that $x_0 > 0$. Then the following cases hold:
\begin{enumerate}
    \item[(i)] If $\gamma_-(\infty) \leq p < \gamma_-(0)$, then \eqref{eq: Volterra Feller boundary condition} holds for all $t \in (0,t_0)$, where $t_0 \in (0,\infty]$ solves the equation $\gamma_-(t_0) = p$.
    \item[(ii)] Condition \eqref{eq: Volterra Feller boundary condition} holds for all $t>0$, provided that $\gamma_-(\infty) > p$.
\end{enumerate}
Finally, since $\lim_{t \searrow 0} \theta(t)^{-1}C(t)^{-1} = x_0^{-1} > 0$, and $\lim_{t \searrow 0} \mathrm{e}^{-\frac{C(t)}{2}} \theta(t)^{-p} = 0$, it follows that $\sup_{t\in[0,T]} \E[X_t^{-p}] < \infty$ whenever \eqref{eq: Volterra Feller boundary condition} holds on $[0,T]$. 
\end{Remark}

The lower bounds in Theorem \ref{thm: boundary regular case} provide a necessary condition for the finiteness of negative moments. Indeed, if $\gamma_+(t) \leq p$, then the lower bounds are not convergent which implies that $\E[X_t^{-p}] = \infty$. Therefore, a necessary condition for the existence of a negative moment of order $p$ is $\gamma_+(t) > p$. Our next corollary provides a characterisation of finite negative moments on intervals $(0,t_0)$ with $t_0 > 0$. 
\begin{Corollary}
    Suppose that $K \in C^1(\R_+)$ and let $p > 0$. Then there exists $t_0 > 0$ such that
    \[
        \E[X_t^{-p}] < \infty \ \forall t \in (0,t_0) \ \Longleftrightarrow \ \frac{2(x_0 L_0(0) + b)}{\sigma^2 K(0)} > p.
    \]
\end{Corollary}
\begin{proof}
    Since $\gamma_+(0) = \gamma_-(0) = \frac{2(x_0 L_0(0) + b)}{\sigma^2 K(0)}$, our necessary and sufficient conditions coincide at $t = 0$. The assertion now follows from the continuity of $t \longmapsto \gamma_{\pm}(t)$.
\end{proof} 

It is worth noting that while the equivalence holds on $(0,t_0)$, our results do not cover $p \in [\gamma_-(t), \gamma_+(t))$ for a fixed $t > 0$, leaving a memory-induced gap in characterizing negative moments. This is a natural consequence of the fact that Volterra models can be viewed as higher- or infinite-dimensional Markovian models \cite{MR3934104, FGW25}, where similar analytical gaps are a known phenomenon (see \cite[Appendix B]{BP24}).

Our last result for the regular case concerns the hitting time of zero for the process with a regular Volterra kernel, and complements the recent results of \cite{BP24}, where this problem was studied for a general class of stochastic Volterra processes with regular Volterra kernels.

\begin{Theorem}[Hitting zero]\label{thm: hitting zero regular case}
    Suppose that $K \in C^1(\R_+)$, and let $T \in (0,\infty]$ satisfy 
    \[
        \gamma_-(T) = \frac{2\bigl(b + x_0L_0(T)\bigr)}{\sigma^2 K(0)} \geq 1.
    \]
    Let $\tau_0 = \inf\{ t > 0 : X_t = 0\}$ denote the hitting time of zero. Then $\P[\tau_0 \geq T] = 1$.
\end{Theorem}

The proof of this theorem is given at the end of Section 4, and is based on a comparison with a CIR process with time-dependent coefficients studied in \cite{MR1368374}. When $x_0 = 0$, this result coincides with that for the classical CIR process after rescaling the drift and volatility according to $b \mapsto K(0)b$ and $\sigma \mapsto K(0)\sigma$. However, for $x_0 > 0$, the memory term contributes the additional positive quantity $x_0L_0(T)$, which pushes the process away from the boundary. The case $T=\infty$ reflects non-attainment of the boundary and reads as $\P[\tau_0 = \infty] = 1$.

\subsection{Rough case}

Next, we focus on the rough case where $K$ is regularly varying as in~\eqref{eq: regularly varying}. In this regime, the process has significantly rougher sample paths, and it is natural to expect that it might hit zero in finite time. Our next result confirms this conjecture and shows that the rough square-root process even has an atom at zero.

\begin{Theorem}[Rough case]\label{Theorem: atom rough}
    Suppose that $K$ is regularly varying at zero, i.e., \eqref{eq: regularly varying} holds with $\alpha \in (1/2,1)$ and a slowly varying function $\ell > 0$. Define the constants
    \begin{align*}
        C_*(\alpha) = 2(2\alpha-1)\frac{\Gamma(1-\alpha)}{\Gamma(2-2\alpha)} \qquad \text{ and }\qquad C^*(\alpha) = \frac{\Gamma(3\alpha-1)}{2\Gamma(2\alpha-1)},
    \end{align*}
    and the time-dependent function
    \begin{align*}
        \Lambda_{\alpha}(t) = x_0 \frac{\Gamma(1-\alpha)}{\Gamma(2 - 2\alpha)}\frac{t^{1-2\alpha}}{\ell(t)^2} + \frac{b}{1-\alpha}\frac{t^{1-\alpha}}{\ell(t)}.
    \end{align*}
    Then for any $\varepsilon \in (0,1)$, there exists $t_0 > 0$ such that for all $t \in (0,t_0)$, the atom at zero satisfies the local bounds
    \begin{align*}
        \exp\bigg( - (1-\varepsilon)\frac{C^*(\alpha)\Lambda_{\alpha}(t)}{\sigma^2} \bigg) \geq \mathbb{P}[X_t = 0] \geq \exp\left( - (1+\varepsilon) \frac{C_*(\alpha)\Lambda_{\alpha}(t)}{\sigma^2} \right) > 0.
    \end{align*}
   For the pure fractional kernel $K(t) = \frac{t^{\alpha-1}}{\Gamma(\alpha)}$, we obtain for each $t > 0$ the global lower bound
    \begin{equation}\label{eq:glob l bd}
        \mathbb{P}[X_t = 0] \geq \exp\left( - \frac{C_*(\alpha)\Lambda_{\alpha}(t)}{\sigma^2} - \frac{2\beta}{\sigma^2}\mathbbm{1}_{\beta > 0} \left[ x_0 \frac{t^{1-\alpha}}{\Gamma(2-\alpha)} + b t \right] \right),
    \end{equation}
    where $\ell(t) \equiv 1$ in the evaluation of $\Lambda_\alpha(t)$. Finally, if $\beta = 0$, then we have the exact equality
    \[
        \mathbb{P}[X_t = 0] = \exp\left( - \frac{C_*(\alpha)\Lambda_{\alpha}(t)}{\sigma^2} \right).
    \]
\end{Theorem}

Note that $\lim_{\alpha \searrow 1/2} C_*(\alpha) = \lim_{\alpha \searrow 1/2} C^*(\alpha) = 0$. Since $\Lambda_\alpha(t)$ remains finite for any fixed $t > 0$ as $\alpha \searrow 1/2$, the exponents in the local bounds vanish. Consequently, this shows that the global lower bound in \eqref{eq:glob l bd} converges to $1$ as $\alpha \searrow 1/2$, provided $\beta \leq 0$ and $t>0$ is fixed. As an illustration, the following table provides numerical values for the upper and lower bounds on the probability of an atom in the purely fractional case. 

\begin{table}[htpb]
    \centering
    \renewcommand{\arraystretch}{1.2}
    \begin{tabular}{@{}llccccc@{}}
        \toprule
        $\alpha$ & Bound & $t = 0.1$ & $t = 0.3$ & $t = 0.5$ & $t = 0.7$ & $t = 0.9$ \\ 
        \midrule
        0.55 & UB & 0.949062 & 0.944872 & 0.941212 & 0.938043 & 0.935209 \\ 
             & LB & 0.775435 & 0.758921 & 0.744729 & 0.732599 & 0.721928 \\ \addlinespace
        0.65 & UB & 0.816001 & 0.832152 & 0.833230 & 0.831647 & 0.829263 \\ 
             & LB & 0.257341 & 0.293327 & 0.295865 & 0.292151 & 0.286601 \\ \addlinespace
        0.75 & UB & 0.609351 & 0.683584 & 0.700832 & 0.706943 & 0.709012 \\ 
             & LB & 0.019009 & 0.047679 & 0.058197 & 0.062382 & 0.063858 \\ \addlinespace
        0.85 & UB & 0.312860 & 0.450271 & 0.486558 & 0.501642 & 0.508931 \\ 
             & LB & $5.58 \times 10^{-5}$ & 0.001201 & 0.002307 & 0.002984 & 0.003370 \\ \addlinespace
        0.95 & UB & 0.035052 & 0.086643 & 0.102804 & 0.109831 & 0.113361 \\ 
             & LB & $7.64 \times 10^{-13}$ & $1.43 \times 10^{-9}$ & $5.94 \times 10^{-9}$ & $1.03 \times 10^{-8}$ & $1.34 \times 10^{-8}$ \\ 
        \bottomrule
    \end{tabular}
    \caption{Global lower (LB) and upper (UB) bounds for the probability $\mathbb{P}[X_t = 0]$ using the pure fractional kernel with parameters $b=0.02$, $\beta=-0.3$, $\sigma=0.3$, and $x_0=0.02$. The same parameters have been used in \cite{MR3934104} and \cite{MR4521278}.}
    \label{tab:bounds_atom}
\end{table}

As a corollary, we immediately see that the Volterra square-root process with a regularly varying kernel necessarily hits zero with positive probability on any interval of the form $(0,t_0)$ for sufficiently small $t_0 > 0$.

\begin{Corollary}
    Suppose that $K$ is regularly varying, that is, \eqref{eq: regularly varying} holds with $\alpha \in (1/2,1)$ and a slowly varying function $\ell > 0$. Let $\tau_0 = \inf\{ t > 0 : X_t = 0 \}$ be the hitting time of zero. Then for each $t \in (0,t_0)$ with $t_0$ given as in Theorem \ref{Theorem: atom rough}, we find
    \[
        \P[\tau_0 \leq t] \geq \P[X_{t} = 0] > 0.
    \]
\end{Corollary}

In contrast to the case of regular kernels, in the rough regime there is \textit{no Feller condition} that classifies boundary attainment. 

\subsection{Equivalent martingale measures in the Volterra Heston model}

Let us model a risky asset $S$ under the real-world measure~$\P$ by 
\begin{align}\label{eq: volterra heston under P}
    \frac{\mathrm{d}S_t}{S_t} = \mu_t\,\mathrm{d}t + \sqrt{X_t}\,\mathrm{d}W_t,
\end{align}
where $\mu_t$ is the (relative) drift of the asset, $X$ denotes the Volterra
square-root process~\eqref{eq: VCIR}, and~$W$ is a Brownian motion with $\mathrm{d}\langle W,B\rangle_t = \rho\,\mathrm{d}t$ with correlation coefficient $\rho\in[-1,1]$.
Let~$r$ denote the risk-free interest rate. Let $\Q \sim \P$ be a risk-neutral pricing measure obtained from Girsanov's theorem via 
\[
    W_t^{\Q} := W_t + \int_0^t \theta^1_s\,\mathrm{d}s 
    \qquad \text{and} \qquad 
    B_t^{\Q} := B_t + \int_0^t \theta^2_s\,\mathrm{d}s,
\]
such that $W^{\Q}$ and $B^{\Q}$ are standard Brownian motions under $\Q$ with the same correlation structure. Here, the progressively measurable processes $(\theta^1_t,\theta^2_t)_{t\ge 0}$ represent the effective market prices of risk. 
Under standard assumptions, \emph{all} equivalent probability measures
are of this form (see Theorem~5.4.1 in~\cite{Wi06}). For~$\mathbb Q$ to be
a local martingale measure, the asset price dynamics under~$\Q$ must take the risk-neutral form
\[
    \frac{\mathrm{d}S_t}{S_t} = r\, \mathrm{d}t + \sqrt{X_t}\, \mathrm{d}W_t^{\Q},
\]
which relates $\theta^1$ and $\mu_t$ via
\begin{equation}\label{eq:th mu}
  \theta_t^1 = \frac{\mu_t-r}{\sqrt{X_t}}.
\end{equation}
The choice of $\theta^2$ is not further restricted. Among the possible choices of $(\theta^1,\theta^2)$ the affine specification
\begin{equation}\label{eq: affine market price risk}
    \theta^1_t = \frac{\lambda_0}{\sqrt{X_t}} + \lambda_1\sqrt{X_t},
    \qquad
    \theta^2_t = \frac{\eta_0}{\sqrt{X_t}} + \eta_1\sqrt{X_t},
\end{equation}
with constants $\lambda_0,\lambda_1,\eta_0,\eta_1$, plays a central role, since it preserves the affine structure of the dynamics. Then, \eqref{eq:th mu} becomes
\[
  \mu_t = r + \lambda_0 + \lambda_1 X_t.
\]
Due to the particular form of \eqref{eq: affine market price risk}, the variance process under $\Q$ retains the form \eqref{eq: VCIR} with risk-adjusted coefficients
\[
    \widetilde{b} = b - \sigma\eta_0 \quad \text{ and } \quad \widetilde{\beta} = \beta - \sigma\eta_1. 
\]
Such a transformation is only meaningful under the premise that the Girsanov change of measure is well-defined. A necessary condition is that $\int_0^T \bigl(|\theta^1_s| + |\theta^2_s|\bigr)\,\mathrm{d}s < \infty$ holds $\P$-a.s. For the 
choice~\eqref{eq: affine market price risk}, with arbitrary~$\lambda_0$ and~$\eta_0$, this requires that
\begin{equation}\label{eq:int X 1/2}
    \int_0^T X_s^{-1/2}\,\mathrm{d}s < \infty
    \qquad \text{a.s.}
\end{equation}
For regular kernels, this holds whenever the boundary condition \eqref{eq: Volterra Feller boundary condition} is satisfied with $p=\tfrac12$. For \textit{rough kernels}, however, this condition \textit{fails to hold}. Indeed, let $L_t = \int_0^t \mathbbm{1}_{\{X_s = 0\}}\, \mathrm{d}s$ and note that $\E[L_t] = \int_0^t \P[X_s = 0]\,\mathrm{d}s$. By Theorem \ref{Theorem: atom rough}, we have $\P[X_s = 0] > 0$ for all $s \in (0,t_0)$, and hence $\E[L_t] > 0$. Since $L_t$ is nonnegative, it follows that $\P[L_t > 0] > 0$. On the event $\{L_t > 0\}$, the set $\mathcal{Z}_t = \{s \in [0,t] : X_s = 0\}$ has positive Lebesgue measure. Consequently, we find $\int_0^t X_s^{-1/2}\,\mathrm{d}s = \infty$ on the event $\{L_t > 0 \}$.

Thus, we necessarily have the severe restriction $\lambda_0 = \eta_0 = 0$ in \eqref{eq: affine market price risk}. In particular, an affine change of measure cannot alter the long-term mean of the variance process.
The resulting $\mathbb P$-drift, i.e.\ $\mu_t = r + \lambda_1 X_t$ with constant~$\lambda_1$,
is assumed in~\cite{HaWo21}, which is one of the few papers that consider
the Volterra Heston model under the real-world measure. Note that, for equity models, this specification is problematic. For $\lambda_1>0$, the drift increases
with volatility, contradicting the leverage effect. For $\lambda_1<0$, the model claims that the equity risk premium cannot be positive. The latter means that a long position in the underlying
is unattractive, which undermines the crucial assumption that the underlying can be liquidly traded.

\begin{Proposition}\label{prop:no elmm}
    A rough Volterra Heston model~\eqref{eq: volterra heston under P} with constant drift~$\mu \neq r$ under~$\mathbb P$
admits no equivalent local martingale measure. 
\end{Proposition}
Indeed, by~\eqref{eq:th mu}, this result follows again from the failure of~\eqref{eq:int X 1/2}. Even more so, to avoid this condition, admissible natural choices for the drift are $\mu_t = r + h(X_t)$ where $h$ is continuous and $h(x)/\sqrt{x}$ is has a continuous extension onto $x=0$. 

While the classical Heston model has well-known drawbacks in comparison to the rough Heston model, the situation regarding equivalent martingale measures
is completely different, as discussed in~\cite{DeLeRo21}.
For instance, under the Feller condition, an equivalent local
martingale measure exists if the drift under~$\mathbb P$
is an arbitrary continuous function of the variance process, e.g.\ a constant (see Theorem~3.7 in~\cite{DeLeRo21}).

As another implication, let us consider a payoff at time $T > 0$ represented by the $\F_T$-measurable random variable $V_T \in L^2(\Omega)$. Suppose that the price process is a square-integrable martingale. The GKW-decomposition (as used e.g.\ in variance-optimal hedging) provides a decomposition with respect to the price process of the form
\[
    \E[V_T\ | \mathcal{F}_t] = \E[V_T] + \int_0^t \vartheta_s\, \mathrm{d}S_s + L_t,
\]
where $\int\vartheta dS$ is the hedgable part, and $L_t$ satisfies $\langle S, L \rangle = 0$ and represents the residual error that is not hedgable by $S$. The replication strategy $\vartheta$ belongs to the Hilbert space with norm
\[
    \E\left[ \int_0^T \vartheta_t^2\, \mathrm{d}\langle S\rangle_t \right] < \infty
\]
and can be explicitly computed from $\vartheta_t = \mathrm{d}\langle H, S\rangle_t/ \mathrm{d}\langle S \rangle_t$ where $H_t = \E[V_T  | \mathcal{F}_t]$. In particular, $\vartheta$ is unique up to $\mathrm{d}\langle S\rangle_t \mathrm{d}\P$ equivalence. Since $\mathrm{d}\langle S\rangle_t \mathrm{d}\P = X_t S_t^2 \mathrm{d}t \mathrm{d}\P$
for the stochastic volatility model with variance process~$X$, we find that each element of the set $\{ \vartheta_t + \phi_t \mathbbm{1}_{X_t = 0} \ : \ \phi \text{ adapted } \}$ determines the same equivalence class and hence provides the same hedging error. For the particular case of the Volterra Heston model with regularly varying kernel $K$, this collection of equivalence classes is nontrivial since $\Omega_0 = \{ (t,\omega) \in [0,T] \times \Omega : X_t(\omega) = 0\}$ satisfies by Fubini's Theorem
\[
\P[\Omega_0] = \int_0^T \P[X_t = 0]\, \mathrm{d}t > 0.
\]
Hence, the hedging strategy is not uniquely determined on events of positive probability.

Further consequences of the atom at zero are left for future research. For example,
the results of~\cite{MaHiJa17} could be used to analyze the effect of the atom on realized variance options.

\subsection{Boundary behaviour of the limit distribution}

Finally, we study the boundary behaviour of the limit distribution for the Volterra square-root process. According to \cite{FJ22}, the limit distribution exists whenever the process is subcritical. A characterisation can be given in terms of the asymptotic behaviour of its first moment. The latter is characterised by the resolvent $E_{\beta} \in L_{\mathrm{loc}}^2(\R_+)$ defined as the unique solution of the linear Volterra equation
\begin{align}\label{eq: Ebeta definition}
        E_{\beta}(t) = K(t) + \beta \int_0^t K(t-s)E_{\beta}(s)\, \mathrm{d}s.
\end{align}
Since $K$ is locally square integrable by assumptions, Young's inequality implies that $K \ast E_{\beta}$ is continuous on $\R_+$, and hence $E_{\beta}$ is continuous on $(0,\infty)$. Moreover, it can be shown that $E_{\beta}$ is nonnegative. 

\begin{Definition}
    The Volterra square-root process is called subcritical, if 
    \[
        \beta < \left( \int_0^{\infty} K(t)\, \mathrm{d}t \right)^{-1}
    \]
    with the convention $(\int_0^{\infty} K(t)\, \mathrm{d}t)^{-1} = 0$, whenever $K \not \in L^1(\R_+)$.
\end{Definition}
Such a condition is always satisfied, if $\beta \leq 0$. When $K \in L^1(\R_+)$, it is equivalent to $\beta \int_0^{\infty}K(t) < 1$, which also reads as $\beta < L_0(\infty)$, see \eqref{eq: 7}. Lemma \ref{lemma: subcriticality} from the appendix shows that subcriticality is equivalent to if $E_{\beta} \in L^1(\R_+)$. For $\beta < L_0(\infty)$, this lemma shows that $E_{\beta} \in L^2(\R_+)$ is completely monotone, and strictly positive. Hence we can apply \cite[Theorem 5.3, Theorem 5.4]{FJ22}, and there exists a unique probability measure $\pi_{x_0}$ on $\R_+$ with finite moments of all orders such that $X_t \Longrightarrow \pi_{x_0}$ weakly as $t \to \infty$. For $u \in \C_-$, its Fourier-Laplace transform is given by 
\begin{equation}\label{eq: pi affine formula}
    \int_{\R_+} \mathrm{e}^{ux}\,\pi_{x_0}(\mathrm{d}x)
    = \exp\left( x_0 u + x_0 \int_0^{\infty} R(\psi(s; u))\, \mathrm{d}s
        + b \int_0^{\infty} \psi(s;u)\, \mathrm{d}s \right),
\end{equation}
where $\psi(\cdot;u) = \psi(t; u \delta_0)$ is given by \eqref{eq: Volterra Riccati measure}, and it was shown that $\psi(\cdot; u) \in L^1(\R_+) \cap L^2(\R_+)$, so that all integrals are well-defined. As a consequence of memory, the limit distribution $\pi_{x_0}$ depends on the initial state $x_0$ if and only if $K \in L^1(\R_+)$. The next theorem provides a characterisation of its boundary behaviour.

\begin{Theorem}[Boundary of the limit distribution]\label{thm: limit_distribution}
    Assume that the Volterra square-root process is subcritical. Then the limit distribution $\pi_{x_0}$ satisfies the following:
    \begin{enumerate}
        \item[(a)] Suppose that $K \in C^1(\R_+)$. Then for each $0 < p < \gamma_-(\infty)$
        \[
            \frac{1}{\theta(\infty)^{p}} \frac{\Gamma( \gamma_+(\infty) - p)}{\Gamma(\gamma_+(\infty))} \leq \int_{\R_+} x^{-p}\, \pi_{x_0}(\mathrm{d}x) \leq \frac{\Gamma(\gamma_{-}(\infty) - p)}{\Gamma(\gamma_{-}(\infty))} \frac{1}{\theta(\infty)^{p}}.
        \]
        where 
        \[
            \gamma_-(\infty) = \frac{2(x_0 L_0(\infty) + b)}{\sigma^2 K(0)}, \quad \gamma_+(\infty) = \frac{L_0(0) - \beta}{L_0(\infty) - \beta} \gamma_-(\infty), \quad \theta(\infty) = \frac{\sigma^2 K(0)^2}{2 (L_0(0) - \beta)}.
        \]
        If $p \geq \gamma_+(\infty)$, then $\int_{\R_+} x^{-p}\pi(\mathrm{d}x) = \infty$.
        \item[(b)] If $K$ is regulary varying at zero, i.e.\ of the form~\eqref{eq: regularly varying}, then
        \[
            \lim_{t \to \infty}\P[X_t=0] = \pi_{x_0}(\{0\}) > 0. 
        \]
    \end{enumerate}
\end{Theorem} 

Since $\pi_{x_0}$ is absolutely continuous with respect to the Lebesgue measure on $(0,\infty)$ (see \cite{FJ22}), it follows that $\pi_{x_0}(\mathrm{d}x) = \pi_{x_0}(\{0\})\delta_0(\mathrm{d}x) + p_{x_0}(x)\, \mathrm{d}x$, where $p_{x_0}$ is the density of the limit distribution. For regular kernels $K \in C^1(\R_+)$ with $x_0L_0(\infty) + b > 0$ (e.g. if $b > 0$ or $x_0 > 0$ and $K \in L^1(\R_+)$), we find $\pi_{x_0}(\{0\}) = 0$, and hence the limit distribution is absolutely continuous with respect to the Lebesgue measure on $\R_+$. However, in the rough case, the limit distribution always has an atom at zero. 

Moreover, while Theorem \ref{Theorem: atom rough} shows that $\P[X_t = 0] > 0$ holds for sufficiently small $t$ and any regularly varying kernel at zero, Theorem \ref{thm: limit_distribution}.(b) implies that $\P[X_t=0] > 0$ also holds for sufficiently large $t$. We expect that the atom is also present for intermediate values $t \in [\delta, \delta^{-1}]$ with $\delta > 0$. The latter is justified for the fractional kernel as shown in Theorem \ref{Theorem: atom rough} and \cite{BJ26}.

\subsection{Structure of the work}

This work is organised as follows. In Section 2, we first prove a regularisation result for convolutions with a Volterra kernel, and then establish Theorem \ref{Theorem: regularity} via a bootstrapping argument. Section 3 lays out the core methods of this work. Specifically, we derive upper bounds for the maximal solution of the corresponding Volterra Riccati equation and then obtain analogous lower bounds. Based on these upper and lower bounds, our main results are proved in Section 4. The appendix contains auxiliary results and, in particular, a comparison theorem for generalized Riemann-Liouville differential equations.

\section{Local regularity for the Volterra Riccati equation}

In this section, we show that the unique solution of \eqref{eq: Volterra Riccati measure} is smooth on $(0,\infty)$. Here and below, we let $C^{\eta}([a,b])$ denote the H\"older space of $\eta \in (0,1)$ H\"older continuous functions on $[a,b]$. Given $m \geq 1$, we let $C^{m, \eta}([a,b])$ be the space of $m$-times continuously differentiable functions on $[a,b]$ such that $f^{(m)} \in C^{\eta}([a,b])$. In this section, we let $K \in L_{\mathrm{loc}}^1(\R_+)$, and for $g \in L_{\mathrm{loc}}^1(\R_+)$, we define the convolution operator
\[
(\mathcal K g)(t):=\int_0^t K(t-s)g(s)\,\mathrm{d}s,\qquad t\in(0,T].
\]
Although Volterra kernels that satisfy (K) are of our primary interest, the results of this section do not explicitly use assumption (K). The following lemma establishes the core smoothing property of convolutions. 

\begin{Lemma}[Global smoothing by convolutions]\label{lem: global smoothing}
Let $T > 0$, $K \in C^{\infty}((0,T])$ and assume that there exists $\rho \in (0,1)$ such that for every integer $m\ge 0$ there exists a constant $C_m > 0$ with the property
\begin{align}\label{eq: smoothing kernel}
|K^{(m)}(t)| \le C_m\, t^{\rho-1-m}, \qquad t \in (0,T].
\end{align} 
Let $\eta \in [0,1)$ such that $\eta+\rho \notin \mathbb{N}$ and $m \ge 0$. Then for each function $f \in C^{m, \eta}([0,T])$ satisfying $f(0) = f'(0) = \dots = f^{(m)}(0) = 0$, the convolution $\mathcal{K}f$ has the Hölder regularity
\[
\mathcal{K}f \in \begin{cases} 
C^{m,\eta+\rho}([0,T]), & \quad\text{if }\eta+\rho<1, \\ 
C^{m+1,\eta+\rho-1}([0,T]), & \quad\text{if } \eta+\rho>1. 
\end{cases}
\]
\end{Lemma}
\begin{proof}
    Let us first consider the case $m=0$. Let $0 \le s < t \le T$. By adding and subtracting $f$ inside the integrals, we find 
    \begin{align*}
        \mathcal{K}f(t) - \mathcal{K}f(s) 
        &= \int_0^s [K(t-v) - K(s-v)][f(v) - f(s)]\,\mathrm{d}v 
        \\ &\quad + \int_s^t K(t-v)[f(v) - f(s)]\,\mathrm{d}v + f(s)\int_s^t K(u)\,\mathrm{d}u =: I_1 + I_2 + I_3.
    \end{align*}
    To simplify the notation, let us define $\Delta := t-s$.

    \textit{Case 1: $\eta+\rho < 1$.} 
    Since $f(0) = 0$, we find $|f(s)| \lesssim s^\eta \leq u^{\eta}$, and hence obtain
    \[
        |I_3| \lesssim C_0 \int_s^t u^{\eta+\rho-1}\,\mathrm{d}u \lesssim \frac{C_0}{\eta+\rho} (t^{\eta+\rho} - s^{\eta+\rho}) \lesssim \Delta^{\eta+\rho}.
    \]
    For $I_2$, we use the local Hölder bound $|f(v)-f(s)| \lesssim (t-s)^\eta$ for $v \in [s, t]$ to find
    \[
        |I_2| \lesssim \Delta^\eta \int_s^t (t-v)^{\rho-1}\,\mathrm{d}v \lesssim \Delta^{\eta+\rho}.
    \]
    For $I_1$, we split the domain of integration at $s-\Delta$. If $v \in [\max(0, s-\Delta), s]$, we use the triangle inequality $|K(t-v) - K(s-v)| \le |K(t-v)| + |K(s-v)|$ to obtain
    \[
        \int_{\max(0, s-\Delta)}^s \left(|K(t-v)| + |K(s-v)|\right) (s-v)^\eta \,\mathrm{d}v \lesssim \int_0^{\Delta \wedge s} u^{\rho+\eta-1}\,\mathrm{d}u \lesssim \Delta^{\eta+\rho}.
    \]
    If $s > \Delta$, then for $v \in [0, s-\Delta]$, we apply the Mean Value Theorem to find $|K(t-v) - K(s-v)| \lesssim \Delta (s-v)^{\rho-2}$. Integrating this against the Hölder bound for $f$ yields
    \[
        \int_0^{s-\Delta} \Delta(s-v)^{\rho+\eta-2}\,\mathrm{d}v \lesssim \Delta \int_{\Delta}^s u^{\rho+\eta-2}\,\mathrm{d}u \lesssim \Delta^{\eta+\rho}.
    \]
    Summing the bounds for $I_1, I_2$, and $I_3$ shows that $\mathcal{K}f \in C^{\eta+\rho}([0,T])$ when $\eta+\rho < 1$.

    \textit{Case 2: $1 < \eta+\rho < 2$.} In this case, let us define the function
    \[
        h(t) = K(t)f(t) + \int_0^t K'(t-v)[f(v) - f(t)]\,\mathrm{d}v,
    \]
    which is well-defined by the regularity of $f$ combined with \eqref{eq: smoothing kernel}. Let $\varepsilon \in (0,T)$ be arbitrary. By similar approximation arguments to \cite[Lemma 2.1]{BFK24}, one can verify that
    \[
        \mathcal{K}f(t) = \mathcal{K}f(\varepsilon) + \int_{\varepsilon}^t h(s)\, \mathrm{d}s, \qquad t \in [\varepsilon, T].
    \]
    Hence $\mathcal{K}f$ is absolutely continuous, and it suffices to show that $h \in C^{\eta + \rho - 1}([0,T])$. The increments satisfy
    \begin{align*}
        h(t) - h(s) &= K(t)f(t) - K(s)f(s) + \int_0^s (K'(t-v) - K'(s-v))(f(v) - f(s))\, \mathrm{d}v 
        \\ &\qquad + \int_s^t K'(t-v)(f(v) - f(t))\, \mathrm{d}v - (f(t) - f(s))\int_0^s K'(t-v)\, \mathrm{d}v.
    \end{align*}
    The second and third terms are identical to $I_1,I_2$ with $K$ replaced by $K'$, and hence have the desired regularity. For the first term, we obtain 
    \begin{align*}
        |K(t)f(t) - K(s)f(s)| &\le |K(t)|\,|f(t)-f(s)| + |f(s)|\,|K(t)-K(s)| 
        \\ &\lesssim t^{\rho-1} \Delta^{\eta} + s^{\eta}\int_s^t |K'(v)|\,\mathrm{d}v \lesssim \Delta^{\eta+\rho - 1},
    \end{align*}
    where we have used $|f(s)| \lesssim s^{\eta}$ and $s^{\eta}\int_s^t |K'(v)|\, \mathrm{d}v \leq \int_s^t v^{\eta}|K'(v)|\, \mathrm{d}v \lesssim \int_s^{t} v^{\eta + \rho - 2}\, \mathrm{d}v \lesssim \Delta^{\eta + \rho - 1}$ by \eqref{eq: smoothing kernel} and since $1 < \eta + \rho < 2$. The last term can be estimated by
    \begin{align*}
        |f(t) - f(s)|\left| \int_0^s K'(t-v)\, \mathrm{d}v\right| 
        &= |f(t)-f(s)| \left| \int_{t-s}^{t}K'(v)\, \mathrm{d}v\right| 
        \\ &\lesssim \Delta^{\eta} \int_{t-s}^{t} v^{\rho - 2}\, \mathrm{d}v
        \lesssim \Delta^{\eta} (\Delta^{\rho - 1} - t^{\rho - 1}) \leq \Delta^{\eta + \rho - 1}.
    \end{align*}
    Hence it has the desired regularity. 

    \textit{Case 3: Higher orders $m \ge 1$.} 
    Since $f^{(j)}(0) = 0$ for $j = 1,\dots, m$, we deduce that $\mathcal{K}f$ is $m$-times continuously differentiable with $(\mathcal{K}f)^{(m)}(t) = \int_0^t K(t-v)f^{(m)}(v)\, \mathrm{d}v$. Since $f^{(m)} \in C^{\eta}([0,T])$ by assumption, this reduces the problem to cases 1 and 2 with $m = 0$.
\end{proof}

While this lemma provides a powerful tool for global regularity, in the case of the Volterra Riccati equation, the presence of $K(t)u$, and more generally $K \ast \mu$ typically exclude any smoothness at $t = 0$. Below we provide an analogue that covers local smoothing of Volterra convolutions. 

\begin{Lemma}[Local smoothing by convolutions]
\label{thm:local-smoothing}
Let $T > 0$ and let $K: (0,T] \longrightarrow (0,\infty)$ be $C^\infty$ satisfying \eqref{eq: smoothing kernel}. Let $\delta\in(0,T)$. If $g\in C^{m,\eta}([\delta/2,T]) \cap L^1([0,T])$ for some integer $m\ge 0$ and some $\eta \in [0,1)$, then, assuming $\eta+\rho\notin\mathbb N$,
\[
\mathcal K g \in \begin{cases} C^{m,\eta+\rho}([\delta,T]), & \quad\text{if }\eta+\rho<1,
\\ C^{m+1,\eta+\rho-1}([\delta,T]), & \quad\text{if } \eta+\rho>1. \end{cases}
\]
\end{Lemma}
\begin{proof}
Fix $\delta\in(0,T)$ and set $a:=\delta/2$. We split the convolution into $s\in(0,a)$ and $s\in[a,t]$, i.e.,
\[
(\mathcal K g)(t) = \int_0^a K(t-s)g(s)\, \mathrm{d}s + \int_a^t K(t-s)g(s)\, \mathrm{d}s =: J_0(t) + J_1(t),
\qquad t\in[\delta,T].
\]
Since $t-s\ge a$ for all $s\in(0,a)$ and $t\in[\delta,T]$, the map $[\delta, T] \ni t\longmapsto K(t-s)$ is $C^\infty$ for every $s \in (0,a)$, and all partial derivatives are uniformly bounded on $[\delta,T] \times(0,a)$. Hence, differentiation under the integral yields $J_0\in C^\infty([\delta,T])$.

It remains to prove regularity for $J_1$. Using the change of variables $u=s-a$, we may write
\[
J_1(t)=\int_0^{t-a} K(t-a-u) g(u+a)\, \mathrm{d}u.
\]
Thus, after translating the interval by $a$, it suffices to prove the claim for
\[
    (\mathcal K\varphi)(x) =\int_0^x K(x-u)\varphi(u)\,\mathrm{d}u,  \qquad x\in[0,T-a],
\]
with $\varphi := g(\cdot+a) \in C^{m,\eta}([0,T-a])$. Denote by $P_m$ its Taylor polynomial of degree $m$ at the origin given by
\[
    P_{m}(u):=\sum_{j=0}^{m}\frac{\varphi^{(j)}(0)}{j!}\,u^j,
    \qquad r(u):=\varphi(u)-P_{m}(u).
\]
Using the change of variables $v = x-u$, we obtain for each monomial 
\[
\left( \frac{\mathrm{d}}{\mathrm{d}x}\right)^{j+1}\int_0^x K(x-u)u^j\,\mathrm{d}u = \left( \frac{\mathrm{d}}{\mathrm{d}x}\right)^{j+1} \int_0^x K(v)(x-v)^j\,\mathrm{d}v = j!K(x).
\]
Since $K$ is smooth on $(0, \infty)$, by linearity, it follows that $\mathcal K P_{m}\in C^\infty([a,T-a])$.

It remains to investigate the remainder $\mathcal K r$. By construction of the Taylor polynomial, the remainder $r \in C^{m, \eta}([0, T-a])$ satisfies $r(0) = r'(0) = \dots = r^{(m)}(0) = 0$. Hence, an application of the global smoothing Lemma \ref{lem: global smoothing} gives 
\[
\mathcal K r\in
\begin{cases}
C^{m,\eta+\rho}([0,T-a]), & \eta+\rho<1,\\[2mm]
C^{m+1,\eta+\rho-1}([0,T-a]), & \eta+\rho>1.
\end{cases}
\]
and hence the same regularity on $[a,T-a]$ by restriction. This proves the desired regularity for $\mathcal{K}\varphi = \mathcal{K}P_m + \mathcal{K}r$ on $[a, T-a]$, and shifting the variable back to $t \in [\delta, T]$ establishes the asserted regularity for $J_1$. 
\end{proof}

The next proposition shows that completely monotone kernels that are regularly varying at zero satisfy condition \eqref{eq: smoothing kernel}. 

\begin{Proposition}\label{proposition: smoothing}
    Suppose that $K$ is completely monotone and regularly varying at $t = 0$, i.e.~\eqref{eq: regularly varying} holds for $\alpha \in (0,1)$. Then for any $\rho \in (0, \alpha)$, $K$ satisfies the derivative bounds~\eqref{eq: smoothing kernel}.
\end{Proposition}
\begin{proof}
    Because $K$ is completely monotone on $(0, \infty)$, $K \in C^\infty(0, \infty)$ and that its derivatives alternate, i.e.\ $(-1)^m K^{(m)}(t) \ge 0$ holds for all $m \geq 0$ and $t > 0$. Hence $|K^{(m)}(t)| = (-1)^m K^{(m)}(t)$ is a non-increasing (monotone) function. Applying iteratively the Monotone Density Theorem (see, \cite[Theorem 1.7.2]{BGT87}) to the asymptotics of $K$ at zero, we find 
    \[
        |K^{(m)}(t)| \sim \prod_{j=0}^{m-1} |\alpha-1-j| \ \frac{t^{\alpha-1-m}}{\Gamma(\alpha)} \ell(t), \qquad t \searrow 0, \ m \geq 0.
    \]        
    Let $\rho \in (0, \alpha)$. Define $\varepsilon := \alpha - \rho > 0$. By Potter's bounds for slowly varying functions (see \cite[Theorem 1.5.6]{BGT87}), for any $\varepsilon > 0$ there exists a constant $A > 0$ and a threshold $t_0 > 0$ such that $\ell(t) \le A \, t^{-\varepsilon}$ holds for all $t \in (0, t_0)$. Substituting this bound into the asymptotic equivalence, we find for $t \in (0, t_0)$
    \begin{align*}
        |K^{(m)}(t)| &\lesssim t^{\alpha-1-m} \ell(t) \lesssim t^{\alpha-1-m} t^{-\varepsilon} = t^{\rho-1-m}.
    \end{align*}
    This proves the assertion on $(0,t_0)$. Since $K$ is smooth on $(0, \infty)$, by adjusting the constant, this bound remains valid on any interval $[t_0, T]$, which proves the assertion.
\end{proof}

We now prove Theorem \ref{Theorem: regularity}, using a bootstrapping approach.

\begin{proof}[Proof of Theorem \ref{Theorem: regularity}]
    Let $f(t) := (K \ast \mu)(t)$ denote the source term and $H(t) := R(\psi(t; \mu))$ the nonlinear term. By \cite[Theorem 3.2]{FJ22}, we have $\psi \in L_{\mathrm{loc}}^2(\R_+; \C_-) \cap C((0,\infty); \C_-)$. Since~$R$ is a polynomial of degree two, it follows that $H \in L_{\mathrm{loc}}^1(\R_+; \C) \cap C((0,\infty); \C)$, and the Volterra equation \eqref{eq: Volterra Riccati measure} takes the form  
    \[
        \psi(t) = f(t) + (\mathcal{K} H)(t), \qquad t > 0.
    \]    
    Because $K$ is completely monotone and regularly varying, Proposition \ref{proposition: smoothing} shows that $K$ satisfies the bounds \eqref{eq: smoothing kernel} for the exponent $\rho = \alpha - \varepsilon$. By choosing $\varepsilon > 0$ sufficiently small, we may assume $\rho \in (1/2, \alpha)$. We now iteratively apply the local smoothing result of 
    Lemma~\ref{thm:local-smoothing} on arbitrary compact subintervals of $(0,\infty)$. 
    
    For the base case, since $H$ is continuous, Lemma \ref{thm:local-smoothing} guarantees that $\mathcal{K}H \in C_{\mathrm{loc}}^{\rho}((0, \infty))$. Consequently, the regularity of $\psi = f + \mathcal{K}H$ is at least the minimum of the regularities of~$f$ and~$\mathcal{K}H$.
    Because~$R$ is smooth, $H$ inherits this exact same improved regularity from~$\psi$. 
    
    For the inductive step, whenever $H \in C_{\mathrm{loc}}^{k, \beta}((0, \infty))$, Lemma \ref{thm:local-smoothing} guarantees that $\mathcal{K}H$ belongs to $C_{\mathrm{loc}}^{k, \beta+\rho}((0, \infty))$ (if $\beta+\rho < 1$) or $C_{\mathrm{loc}}^{k+1, \beta+\rho-1}((0, \infty))$ (if $\beta+\rho > 1$). Note that we may always choose $\varepsilon$ such that $j \rho \not \in \mathbb{N}$ for a given number of finite steps. After finitely many iterations, the regularity of the integral term $\mathcal{K}H$ will exceed $m + \eta$. At this point, the regularity of the sum $\psi = f + \mathcal{K}H$ is determined by its least smooth component, which by assumption is $f \in C^{m, \eta}((0,\infty); \C_-)$. 
\end{proof}

\section{Bounds on the Riccati Volterra equation}

\subsection{The rough case}

In this section, we construct the corresponding maximal solution of the Volterra Riccati equation associated with the Laplace transform of the Volterra square-root process, and study its asymptotics. For this purpose, let us define for $u \geq 0$
\begin{equation}\label{eq:def V_u}
    V_u(t) = -\psi(t; -u\delta_0) \geq 0
\end{equation}
with $\psi$ given as in \eqref{eq: Volterra Riccati measure}. One easily checks that $V_u$ is the unique solution to the Volterra equation
\begin{align}\label{eq: Laplace Volterra Riccati equation}
    V_u(t) = K(t)u + \int_0^t K(t-s)\left( \beta V_u(s) - \frac{\sigma^2}{2}V_u(s)^2\right)\, \mathrm{d}s.
\end{align}
In some cases, it is convenient to work with its mild formulation, which absorbs the linear term $\beta V_u$. Using the resolvent $E_{\beta}$ defined in \eqref{eq: Ebeta definition}, the variation of constants formula for Volterra equations shows that \eqref{eq: Laplace Volterra Riccati equation} is equivalent to
\begin{align*}
    V_u(t) = E_{\beta}(t)u - \frac{\sigma^2}{2}\int_0^t E_{\beta}(t-s)V_u(s)^2\, \mathrm{d}s.
\end{align*}
As a first step, we determine the asymptotics of $V_u$ at $t = 0$.

\begin{Lemma}[Asymptotics of $V_u$ at zero]\label{lemma: Vu asymptotics zero}
    Suppose $K$ is regularly varying of the form \eqref{eq: regularly varying} with $\alpha \in (1/2,1)$. Let $u > 0$, then 
    \[
        V_u(t) \sim \frac{u}{\Gamma(\alpha)} t^{\alpha-1}\ell(t), \qquad t \searrow 0.
    \]
\end{Lemma}
\begin{proof}
    Note that, for sufficiently small $t \in (0,t_0)$, the resolvent $E_{\beta}$ admits the Neumann series expansion $E_{\beta} = K + \sum_{j=1}^{\infty}\beta^j K^{\ast (j+1)}$. An application of Potter's bounds yield $|\ell(t)| \lesssim_{\varepsilon} t^{-\varepsilon}$ for given $\varepsilon > 0$, and hence $|K(t)|\lesssim_{\varepsilon} t^{\alpha - 1 - \varepsilon}$ on $(0,t_0)$. Hence the convolution is pointwise upper bounded by the convolution of the fractional kernel $t^{\alpha - 1 - \varepsilon}$, i.e.\ there exists some constant $C > 0$ such that
    \[
        |K^{\ast (j+1)}(t)| \lesssim_{\varepsilon} C^{j+1} \frac{\Gamma(\alpha-\varepsilon)^{j+1}}{\Gamma((j+1)(\alpha-\varepsilon))}t^{(j+1)(\alpha - \varepsilon) - 1}.
    \]
    Since the Gamma function grows super-exponentially, the series $\sum_{j=1}^{\infty}\beta^j K^{\ast (j+1)}$ is uniformly convergent on $(0,t_0)$. The dominant term  corresponds to $j=1$, and so the entire series is of order $\mathcal{O}(t^{2(\alpha-\varepsilon)-1})$. Since $2\alpha - 1 > \alpha - 1$, the series contributes only at higher orders, which yields $E_{\beta}(t) \sim K(t) \sim \frac{t^{\alpha-1}}{\Gamma(\alpha)}\ell(t)$.
    
    Since $V_u \geq 0$, the Volterra equation guarantees $0 \leq V_u(t) \leq E_\beta(t)u$ for all $t > 0$. We use this to bound the nonlinear term:
    \[
        0 \leq \int_0^t E_{\beta}(t-s) V_u(s)^2\, \mathrm{d}s \leq u^2 \int_0^t E_{\beta}(t-s) E_{\beta}(s)^2\, \mathrm{d}s \sim \frac{u^2 B(\alpha, 2\alpha-1)}{\Gamma(\alpha)^3} t^{3\alpha-2}\ell(t)^3,
    \]
    where the asymptotic equivalence follows from Lemma \ref{lemma: convolution regularly varying}. Because $\alpha > 1/2$, the exponent of this error term satisfies $3\alpha - 2 = (\alpha - 1) + (2\alpha - 1) > \alpha - 1$. Therefore, the nonlinear integral can be absorbed into $o(t^{\alpha-1}\ell(t))$, and vanishes asymptotically compared to the linear term. We conclude $V_u(t) \sim \frac{u}{\Gamma(\alpha)} t^{\alpha-1}\ell(t)$.
\end{proof}

A pair $(K,L)$ is called a Sonine pair if it satisfies $K \ast L = 1$, i.e.\ $L$ is the resolvent of the first kind for the Volterra kernel $K$. For such pairs, one can define the generalised Riemann-Liouville derivative by
\[
    D_Lf(t) = \frac{\mathrm{d}}{\mathrm{d}t} \int_{[0,t]} f(t-s)L(\mathrm{d}s), \qquad t > 0,
\]
whenever $f: (0,T] \longrightarrow \R$ is continuous and integrable on $[0,T]$ such that $L \ast f$ is absolutely continuous. For additional details we refer to \cite{math9060594}, while some complementary comparison theorems are collected in Appendix~\ref{se:comp}. 

Note that \eqref{eq: Laplace Volterra Riccati equation} can be rewritten as a fractional differential equation with respect to $D_L$. Convolving \eqref{eq: Laplace Volterra Riccati equation} with $L$ gives
\begin{align}\label{eq: L convolution V}
    (L\ast V_u)(t) = u + \int_0^t \left( \beta V_u(s) - \frac{\sigma^2}{2}V_u(s)^2\right)\, \mathrm{d}s, \qquad t > 0.
\end{align}
Since $V_u \in L_{\mathrm{loc}}^2(\R_+) \cap C^{\infty}(0,\infty)$ by Theorem \ref{Theorem: regularity}, it follows that $t \longmapsto (L \ast V_u)(t)$ is smooth on $(0,\infty)$. Its derivative satisfies the fractional Riccati differential equation on $(0,\infty)$ given by
\begin{align}\label{eq: fractional Riccati differential equation}
    (D_LV_u)(t) =  \beta V_u(t) - \frac{\sigma^2}{2} V_u(t)^2, \qquad \lim_{t \searrow 0}\ (L \ast V_u)(t) = u.
\end{align}
Below we derive upper bounds on $V_u$ uniformly in $u$. For this purpose we use a comparison theorem for \eqref{eq: fractional Riccati differential equation} established in Appendix B. The latter extends \cite{MR4389754} to our specific setting with initial conditions that are not necessarily regular in $t = 0$. 

\begin{Theorem}[Maximal Solution]\label{Theorem: upper bound rough}
    The mapping $u \longmapsto V_u(t)$ is nondecreasing for each $t > 0$. Define the maximal solution
    \[
        V_{\infty}(t) := \sup_{ u \geq 0} V_u(t).
    \]
    Then the following assertions hold:
    \begin{enumerate}
        \item[(a)] If $K$ is regularly varying of the form \eqref{eq: regularly varying} where $\alpha \in (1/2,1)$ and $\ell$ is slowly varying at zero, then for each $\eta \in (0,1)$ there exists $t_0 > 0$ such that
        \[
            V_{\infty}(t) \leq (1+\eta)\frac{2C_{\alpha}}{\sigma^2}\frac{t^{-\alpha}}{\ell(t)}, \qquad t \in (0,t_0),
        \]
        where the constant is given by
        \[
        C_{\alpha} = \frac{4^{\alpha}\sqrt{\pi}}{|\Gamma\left(\frac{1}{2} - \alpha\right)|} = \frac{\Gamma(1-\alpha)}{|\Gamma(1-2\alpha)|}.
        \]
        Moreover, the function $V_{\infty}$ satisfies the fractional Riccati differential equation
        \begin{equation}\label{eq:fRic infty} 
            (D_LV_{\infty})(t) = \beta V_{\infty}(t) - \frac{\sigma^2}{2}V_{\infty}(t)^2, \qquad t > 0.
        \end{equation}
        \item[(b)] For the pure fractional kernel $K(t) = \frac{t^{\alpha-1}}{\Gamma(\alpha)}$, we obtain 
        \[
            V_{\infty}(t) \leq \frac{2C_{\alpha}}{\sigma^2} t^{-\alpha} + \mathbbm{1}_{\beta > 0} \frac{2\beta}{\sigma^2}, \qquad t > 0.
        \]
        with equality when $\beta = 0$.
    \end{enumerate}
\end{Theorem}
\begin{proof}
    To prove the monotonicity in $u$, let $u_1 \leq u_2$ and set $\chi(t) = V_{u_2}(t) - V_{u_1}(t)$. This function satisfies the Volterra equation $\chi(t) = K(t)(u_2 - u_1) + \int_0^t K(t-s)\beta \chi(s)\, \mathrm{d}s - \frac{\sigma^2}{2}\int_0^t K(t-s)(V_{u_2}(s) + V_{u_1}(s))\chi(s)\, \mathrm{d}s$. Since $K$ is completely monotone, and $u_2 - u_1 \geq 0$, an application of \cite[Theorem C.2]{MR4019885} gives $\chi \geq 0$ and thus $V_{u_1} \leq V_{u_2}$. 

    (a) For a given $M > 2C_{\alpha}/ \sigma^2$, let us define $f_{\alpha}(t) = M \frac{t^{-\alpha}}{\ell(t)}$. Lemma~\ref{lemma: first kind asymptotics}~(a) gives $L(t) \sim \frac{1}{\Gamma(1-\alpha)}\frac{t^{-\alpha}}{\ell(t)}$ and hence, by Potter's bound, $L$ satisfies condition \eqref{eq: appendix 1} with $\gamma = 1$. By Lemma~\ref{thm:local-smoothing}, $(L \ast f_{\alpha})(t)$ is smooth on $(0,\infty)$ and hence $f_{\alpha}$ belongs to $S_{L, 1}$ in the notation of Appendix B. Moreover, an application of Lemma \ref{lemma: convolution regularly varying} provides the asymptotics as $t \searrow 0$ 
    \begin{align*}
        (L \ast f_{\alpha})(t) \sim M \frac{\Gamma(1-\alpha)}{(1 - 2\alpha)\Gamma(1-2\alpha)}\frac{t^{1-2\alpha}}{\ell(t)^{2}}.
    \end{align*}
    Since $K$ is regularly varying at $t = 0$, it has the form $K(t) = \frac{t^{\alpha - 1}}{\Gamma(\alpha)}\ell(t)$. Since $K$ is completely monotone and regularly varying, an application of the Monotone Density Theorem shows that $t^n K^{(n)}(t)/K(t) \longrightarrow \prod_{k=1}^{n} (\alpha - k)$. Hence $K$ is smoothly varying. Since $t^{1 - \alpha}$ is smoothly varying, also the product $\ell(t) = \Gamma(\alpha)t^{1-\alpha}K(t)$ is smoothly varying in the sense that 
    \[
        \lim_{t \searrow 0}\frac{t^n \ell^{(n)}(t)}{\ell(t)} = 0, \qquad \forall n \geq 1.
    \]
    Hence also $1/\ell$ is smoothly varying. Moreover, since $L$ is completely monotone and regularly varying with index $-\alpha$ (see Lemma \ref{lemma: first kind asymptotics}), it can be written as $L(t) = t^{-\alpha} \widetilde{\ell}(t)$, where $\widetilde{\ell}$ is smoothly varying by the same argument as given for $\ell$. Hence Lemma \ref{lemma: convolution derivative asymptotics} is applicable, and yields
    \[
        (D_L f_{\alpha})(t) \sim M \frac{\Gamma(1-\alpha)}{\Gamma(1-2\alpha)}  t^{-2\alpha} \ell(t)^{-2} = - M C_{\alpha} t^{-2\alpha}\ell(t)^{-2},
    \]
    where the last equality follows from $C_{\alpha} = -\frac{\Gamma(1-\alpha)}{\Gamma(1-2\alpha)}$ due to Legendre's Duplication Formula. Hence, noting that $-\beta f_{\alpha}(t)$ has only higher-order contributions, we obtain    
    \[
        (D_L f_{\alpha})(t) - \beta f_{\alpha}(t) + \frac{\sigma^2}{2}f_{\alpha}(t)^2 \sim M\left[ \frac{\sigma^2}{2} M  - C_{\alpha} \right] t^{-2\alpha} \ell(t)^{-2}.
    \]
    Since $M > \frac{2C_{\alpha}}{\sigma^2}$, the right-hand side is strictly positive for small $t$. Hence, we find $t_0 > 0$ such that for $t \in (0,t_0)$ 
    \[
        (D_LV_u)(t) - \beta V_u(t) + \frac{\sigma^2}{2}V_u(t)^2 = 0 < (D_L f_{\alpha})(t) - \beta f_{\alpha}(t) + \frac{\sigma^2}{2}f_{\alpha}(t)^2,
    \]
    where the equality stems from \eqref{eq: fractional Riccati differential equation}. Further, using $V_u(t) \sim \frac{u t^{\alpha-1}}{\Gamma(\alpha)}\ell(t)$ by Lemma \ref{lemma: Vu asymptotics zero}, we find $f_{\alpha}(t) - V_u(t) \sim M \frac{t^{-\alpha}}{\ell(t)} - \frac{t^{\alpha-1}}{\Gamma(\alpha)}\ell(t) u > 0$ for sufficiently small $t$ since $\alpha > 1 - \alpha$. An application of Theorem \ref{thm:strict} yields $V_u(t) \leq f_{\alpha}(t)$ for $t \in (0,t_0)$. Taking the limit $u \nearrow \infty$, proves $V_{\infty}(t) \leq f_{\alpha}(t)$. 

    Next, we show that $V_\infty$ solves the fractional Riccati differential equation. Fix $t > 0$ and let $\varepsilon \in (0,t\wedge t_0)$. Using \eqref{eq: L convolution V}, we find 
    \[
        (L \ast V_u)(t) - (L \ast V_u)(\varepsilon) = \int_{\varepsilon}^t \left( \beta V_u(s) - \frac{\sigma^2}{2} V_u(s)^2 \right) \, \mathrm{d}s.
    \]
    Since $u \mapsto V_u(s)$ is monotonically increasing to $V_\infty(s)$, the Monotone Convergence Theorem implies $\lim_{u \to \infty} (L \ast V_u)(x) = (L \ast V_\infty)(x) < \infty$ for any $x > 0$ with the limit being finite due to the previously shown bound $V_{\infty} \leq f_{\alpha}$. On the right-hand side, since $V_u \leq V_\infty$ and $V_\infty$ is bounded on the compact interval $[\varepsilon, t]$, we may apply the Dominated Convergence Theorem to pass the limit $u \to \infty$ inside the integral. This gives
    \begin{align}\label{eq: LVinfty}
        (L \ast V_\infty)(t) - (L \ast V_\infty)(\varepsilon) = \int_\varepsilon^t \left( \beta V_\infty(s) - \frac{\sigma^2}{2} V_\infty(s)^2 \right) \, \mathrm{d}s.
    \end{align}
    Differentiating with respect to $t \geq \varepsilon$, and noting that $\varepsilon > 0$ was arbitrary, proves the assertion. 
    
    (b) Let us first consider the case $\beta = 0$. In this case, the Riccati-Volterra equation for $u > 0$ is given by 
    \begin{equation*}
        V_u(t) = u \frac{t^{\alpha-1}}{\Gamma(\alpha)} - \frac{\sigma^2}{2} \int_0^t \frac{(t-s)^{\alpha-1}}{\Gamma(\alpha)} V_u(s)^2 \, \mathrm{d}s.
    \end{equation*}
    Scaling the time by $\lambda > 0$, i.e.\ $t = \lambda x$, and using the change of variables $s = \lambda \tau$ (hence $\mathrm{d}s = \lambda \, \mathrm{d}\tau$), we get
    \begin{align*}
        V_u(\lambda x) &= u \frac{(\lambda x)^{\alpha-1}}{\Gamma(\alpha)} - \frac{\sigma^2}{2} \int_0^x \frac{(\lambda x - \lambda \tau)^{\alpha-1}}{\Gamma(\alpha)} V_u(\lambda \tau)^2 \lambda \, \mathrm{d}\tau
        \\ &= u \frac{\lambda^{\alpha-1} x^{\alpha-1}}{\Gamma(\alpha)} - \frac{\sigma^2}{2} \lambda^\alpha \int_0^x \frac{(x-\tau)^{\alpha-1}}{\Gamma(\alpha)} V_u(\lambda \tau)^2 \, \mathrm{d}\tau,
    \end{align*}
    Multiplying both sides by $\lambda^{\alpha}$, and setting $W_u^\lambda(x) := \lambda^\alpha V_u(\lambda x)$, shows that 
    \begin{equation} \label{eq:scaled_W}
        W_u^\lambda(x) = u \lambda^{2\alpha-1} \frac{x^{\alpha-1}}{\Gamma(\alpha)} - \frac{\sigma^2}{2} \int_0^x \frac{(x-\tau)^{\alpha-1}}{\Gamma(\alpha)} W_u^\lambda(\tau)^2 \, \mathrm{d}\tau.
    \end{equation}
    By uniqueness, we find $\lambda^\alpha V_u(\lambda t) = W_u^{\lambda}(t) = V_{u \lambda^{2\alpha -1}}(t)$. Taking the limit $u \to \infty$ with $t,\lambda$ fixed shows that the maximal solution satisfies the scaling property $\lambda^{\alpha}V_{\infty}(\lambda t) = V_{\infty}(t)$. Setting $t = 1$ gives $V_{\infty}(\lambda) = V_{\infty}(1)\lambda^{-\alpha}$. It remains to compute the constant $V_{\infty}(1)$. 

    For this purpose, we note that $(D_LV_{\infty})(1) = - \frac{\sigma^2}{2}V_{\infty}(1)^2$ by part (a). The left-hand side is given by
    \begin{align*}
        (D_LV_{\infty})(t) &= \frac{\mathrm{d}}{\mathrm{d}t} (L \ast V_{\infty})(t) \\
        &= \frac{\mathrm{d}}{\mathrm{d}t} \left( \int_0^t \frac{(t-s)^{-\alpha}}{\Gamma(1-\alpha)} V_{\infty}(1) s^{-\alpha} \, \mathrm{d}s \right) \\
        &= V_{\infty}(1) \frac{\mathrm{d}}{\mathrm{d}t} \left( \frac{t^{1-2\alpha}}{\Gamma(1-\alpha)} \int_0^1 (1-y)^{-\alpha} y^{-\alpha} \, \mathrm{d}y \right) \\
        &= V_{\infty}(1) \frac{\mathrm{d}}{\mathrm{d}t} \left( \frac{t^{1-2\alpha}}{\Gamma(1-\alpha)} B(1-\alpha, 1-\alpha) \right) = - C_\alpha V_{\infty}(1) t^{-2\alpha}.
    \end{align*}
    Hence at $t = 1$, we get $-C_{\alpha} V_{\infty}(1) = (D_LV_{\infty})(1) = - \frac{\sigma^2}{2}V_{\infty}(1)^2$, and hence $V_{\infty}(1) = 2C_{\alpha}/\sigma^2$. 

    Now, suppose $\beta < 0$, and denote by $V_u^{(0)}$ the unique solution of \eqref{eq: Laplace Volterra Riccati equation} with $\beta = 0$. Define $\chi(t) = V_u^{(0)}(t) - V_u(t)$. Then
    \[
        \chi(t) = |\beta| \int_0^t K(t-s)V_u(s)\, \mathrm{d}s - \frac{\sigma^2}{2} \int_0^t K(t-s)\left(V_u^{(0)}(s) + V_u(s)\right)\chi(s)\, \mathrm{d}s.
    \]
    An application of \cite[Theorem C.2]{MR4019885} gives $\chi \geq 0$, and hence $V_u(t) \leq V_u^{(0)}(t)$. Thus we obtain $V_{\infty}(t) \leq V_{\infty}^{(0)}(t) = \frac{2C_{\alpha}}{\sigma^2} t^{-\alpha}$ for all $t > 0$.

    Finally, suppose $\beta > 0$. We construct a global supersolution by adding the stationary limit $c = \frac{2\beta}{\sigma^2}$ of the Volterra Riccati equation. Namely, define $g(t) = \frac{2C_{\alpha}}{\sigma^2} t^{-\alpha} + c = V_{\infty}^{(0)}(t) + c$. Noting that $(D_L c)(t) = c L_0(t)$ and $(D_L V_{\infty}^{(0)})(t) = -\frac{\sigma^2}{2}V_{\infty}^{(0)}(t)^2$ as shown for the case $\beta = 0$, we obtain
    \begin{align*}
        (D_L g)(t) - \beta g(t) + \frac{\sigma^2}{2}g(t)^2
        &= -\frac{\sigma^2}{2}V_{\infty}^{(0)}(t)^2 + c L_0(t) - \beta(V_{\infty}^{(0)}(t) + c) 
        \\ &\qquad + \frac{\sigma^2}{2}\left( V_{\infty}^{(0)}(t)^2 + 2c V_{\infty}^{(0)}(t) + c^2 \right) 
        \\ &= c L_0(t) - \beta V_{\infty}^{(0)}(t) - \beta c + \frac{\sigma^2}{2} (2c) V_{\infty}^{(0)}(t) + \frac{\sigma^2}{2}c^2
        \\ &= \frac{2\beta}{\sigma^2} L_0(t) + \beta V_{\infty}^{(0)}(t) > 0
    \end{align*}
    for each $t > 0$ since $V_{\infty}^{(0)}(t) > 0$ for all $t > 0$. An application of the comparison principle Theorem \ref{thm:strict} yields $V_{\infty}(t) \leq g(t) = \frac{2C_{\alpha}}{\sigma^2} t^{-\alpha} + \frac{2\beta}{\sigma^2}$ globally.
\end{proof}

Below we focus on a sharp lower bound for the maximal solution $V_{\infty}$. Using Lemma \ref{lemma: Vu asymptotics zero}, we obtain for each $\eta \in (0,1)$ the lower bound $V_{\infty}(t) \geq V_u(t) \geq (1-\eta)K(t)u = (1-\eta) \frac{t^{\alpha-1}}{\Gamma(\alpha)}\ell(t)u$ for sufficiently small $t \in (0,t_0)$. To match it with its upper bound, we require that (up to constants) $\frac{t^{-\alpha}}{\ell(t)} \sim t^{\alpha-1}\ell(t)u$. This gives $u \sim t^{1-2\alpha} \ell(t)^{-2}$. However, then the lower bound is valid only on $(0,t_0)$ and $t_0 = t_0(u)$ depends on $u$. Our next theorem tracks the dependence of $t_0(u)$ on $u$ and hence provides a rigorous justification of this argument.

\begin{Theorem}[lower bound on the maximal solution]\label{Theorem: lower bound rough}
    Suppose that $K$ is regularly varying of the form \eqref{eq: regularly varying} where $\alpha \in (1/2,1)$ and $\ell$ is slowly varying at zero. Then for each $\eta \in (0,1)$ there exists $t_0 > 0$ such that
    \[
        V_{\infty}(t) \geq (1-\eta)\frac{A_{\alpha}}{2\sigma^2} \frac{t^{-\alpha}}{\ell(t)}, \qquad t \in (0,t_0)
    \]
    where the constant $A_{\alpha}$ is given by
    \[
        A_{\alpha} = \frac{\Gamma(\alpha)}{B(\alpha, 2\alpha-1)} = \frac{\Gamma(3\alpha-1)}{\Gamma(2\alpha-1)}.
    \]
\end{Theorem}
\begin{proof}
    Recall that, since $K(t) = \frac{t^{\alpha-1}}{\Gamma(\alpha)}\ell(t)$ is completely monotone and not identically zero, $\ell > 0$ is smooth on $(0,\infty)$. For each $u > 0$ and a constant $M > 0$ independent of~$u$ to be chosen later, let us define $t_u > 0$ as the generalised inverse
    \[
        t_u = \sup\{ t \in (0,1] \ : \ M t^{1 - 2\alpha}\ell(t)^{-2} \geq u \} \in (0,1].
    \]
    Since $(0,1] \ni t \longmapsto M t^{1-2\alpha}\ell(t)^{-2}$ is regularly varying at zero with index $1 - 2\alpha<0$, it diverges as $t \searrow 0$. Thus, $t_u$ is well-defined, by continuity and the intermediate-value theorem satisfies $M t_u^{1-2\alpha} \ell(t_u)^{-2} = u$ for sufficiently large $u$, and hence $t_u \searrow 0$ as $u \nearrow \infty$. In particular, we have the property
    \begin{align}\label{eq: t0 tu relation}
        \forall t_0 \in (0,1) \ \exists u_0 > 0 \text{ large enough such that } t_u < t_0, \qquad \forall u \geq u_0.
    \end{align}
    Set $I_u = (0,t_u]$ and define the Banach space
    \[
        \mathcal{W}_u = \left\{ f \in C(I_u) \ : \ \exists \lim_{t \searrow 0}\frac{t^{1-\alpha}}{\ell(t)}f(t) \right\} \ \text{ with } \ \| f\|_{u} = \sup_{t \in I_u} \left| \frac{t^{1-\alpha}}{\ell(t)}f(t)\right|.
    \]
    Since pointwise bounds are preserved by pointwise convergence, for any fixed constants $0 < c_1 < 1/\Gamma(\alpha) < c_2$ the set 
    \[
        \mathcal{U}_u = \left\{ f \in \mathcal{W}_u \ : \ c_1 u \leq \frac{t^{1-\alpha}}{\ell(t)}f(t) \leq c_2 u \right\}
    \]
    is a closed subset of $\mathcal{W}_u$, and hence a complete metric space. Define the mapping 
    \[
        G_u(f)(t) = K(t)u + \int_0^t K(t-s)\left( \beta f(s) - \frac{\sigma^2}{2}f(s)^2 \right)\, \mathrm{d}s.
    \]
    Below we prove the existence of $u_0 > 0$ such that $G_u$ is a contraction on $\mathcal{U}_u$ for $u \geq u_0$. Here and below we let $t_0 > 0$ be some small constant that will be chosen sufficiently small at each step, and $u_0$ is the corresponding value determined by \eqref{eq: t0 tu relation}.

    \textit{Step 1. Invariance $G_u(\mathcal{U}_u) \subseteq \mathcal{U}_u$}. 
    Let $f \in \mathcal{U}_u$, then $f(s) \leq c_2 u s^{\alpha-1}\ell(s)$. Using the standard asymptotic properties of convolutions of regularly varying functions (see Lemma \ref{lemma: convolution regularly varying}), for any small $\varepsilon > 0$, there exists $t_0 > 0$ small enough, and by \eqref{eq: t0 tu relation} $u_0$ large enough such that for all $u \geq u_0$
    \begin{align*}
        G_u(f)(t) &\leq \frac{t^{\alpha-1}}{\Gamma(\alpha)}\ell(t)u + c_2 u |\beta| \int_0^t \frac{(t-s)^{\alpha-1}}{\Gamma(\alpha)}\ell(t-s) s^{\alpha-1}\ell(s)\, \mathrm{d}s
        \\ &\leq \frac{t^{\alpha - 1}}{\Gamma(\alpha)}\ell(t)u + (1+\varepsilon) c_2 u |\beta| \frac{B(\alpha, \alpha)}{\Gamma(\alpha)} t^{2\alpha-1}\ell(t)^2
        \\ &= u t^{\alpha-1}\ell(t) \left[ \frac{1}{\Gamma(\alpha)} + |\beta| c_2 (1+\varepsilon) \frac{B(\alpha, \alpha)}{\Gamma(\alpha)} t^\alpha \ell(t) \right].
    \end{align*}
    Since $1/\Gamma(\alpha) < c_2$ and $t^\alpha \ell(t) \to 0$ as $t \searrow 0$, we can adjust $t_0$ smaller and hence $u_0$ larger such that the bracket satisfies for $t \in I_u$
    \[ 
        \frac{1}{\Gamma(\alpha)} + |\beta| c_2 (1+\varepsilon) \frac{B(\alpha, \alpha)}{\Gamma(\alpha)} t^\alpha \ell(t)  < c_2. 
    \]

    For the lower bound, we apply the same bounds and convolution asymptotics to find 
    \begin{align}
        &\ G_u(f)(t) \notag
        \\ &\geq \frac{t^{\alpha-1}}{\Gamma(\alpha)}\ell(t)u - c_2 u |\beta| \int_0^t \frac{(t-s)^{\alpha-1}}{\Gamma(\alpha)}\ell(t-s) s^{\alpha-1}\ell(s)\, \mathrm{d}s \notag
        \\ &\qquad \qquad \qquad \qquad - \frac{\sigma^2}{2}c_2^2 u^2 \int_0^t \frac{(t-s)^{\alpha-1}}{\Gamma(\alpha)}\ell(t-s) s^{2\alpha-2}\ell(s)^2\, \mathrm{d}s \notag
        \\ &\geq u t^{\alpha-1} \ell(t)\left[ \frac{1}{\Gamma(\alpha)} - (1+\varepsilon) c_2 |\beta| \frac{B(\alpha,\alpha)}{\Gamma(\alpha)}t^{\alpha}\ell(t) - \frac{\sigma^2}{2}c_2^2 (1+\varepsilon) u \frac{B(\alpha, 2\alpha-1)}{\Gamma(\alpha)} t^{2\alpha-1}\ell(t)^2 \right]. \label{eq:est Gu}
    \end{align}
    Let us bound the term $u t^{2\alpha-1}\ell(t)^2$. By Potter's bounds, for each $\delta > 0$ small, there exists $t_0 > 0$ such that for all $0 < t \leq t_u < t_0$, and hence by \eqref{eq: t0 tu relation} for some~$u_0$ large enough,
    \[
        \frac{\ell(t)}{\ell(t_u)} \leq (1+\varepsilon) \left( \frac{t}{t_u}\right)^{-\delta}.
    \]
    By definition of the generalized inverse, for $t \leq t_u$, we find $u t_u^{2\alpha-1}\ell(t_u)^2 \leq M$. Hence we obtain for $0 < t \leq t_u$ and $u \geq u_0$
    \begin{align*}
    u t^{2\alpha - 1}\ell(t)^2 &\leq (1+\varepsilon)^2 \left( u t_u^{2\alpha-1}\ell(t_u)^2\right) \left( \frac{t}{t_u}\right)^{2\alpha - 1 - 2\delta} 
    \\ &\leq (1+\varepsilon)^2 \left( u t_u^{2\alpha-1}\ell(t_u)^2\right) 
    \\ &\leq (1+\varepsilon)^2 M,
    \end{align*}
    provided we let $\delta$ be small enough such that $2\alpha - 1 - 2\delta > 0$. Since $1/\Gamma(\alpha) > c_1$, we may choose $\varepsilon$ small enough such that $\frac{1}{\Gamma(\alpha)} > \varepsilon + c_1$. For the second term in the bracket
    in~\eqref{eq:est Gu}, noting that $t^{\alpha}\ell(t) \to 0$ as $t \to 0$, we may further decrease $t_0$ and by \eqref{eq: t0 tu relation} increase $u_0$ to bound the term corresponding to the linear term of the VIE by~$\varepsilon$. Hence, we obtain
    \begin{align}\label{eq: 8}
        G_u(f)(t) &\geq u t^{\alpha-1}\ell(t)\left[ \frac{1}{\Gamma(\alpha)} - \varepsilon -  (1+\varepsilon)^3\frac{\sigma^2}{2} c_2^2 A^{-1}_{\alpha} M\right] > c_1 u t^{\alpha-1}\ell(t),
    \end{align}
    provided that $M$ satisfies
    \[
        0 < M \leq (1+\varepsilon)^{-3}\frac{2A_{\alpha}}{\sigma^2 c_2^2} \left( \frac{1}{\Gamma(\alpha)} - \varepsilon - c_1 \right).
    \] 
    This shows that $G_u(\mathcal{U}_u) \subseteq \mathcal{U}_u$.

    \textit{Step 2. Strict Contraction.}
    For $f, g \in \mathcal{U}_u$, the difference satisfies
    \[
        |f(t) - g(t)| \leq t^{\alpha-1}\ell(t)\| f - g\|_u, \qquad t \in I_u,
    \]
    while the difference of squares gives 
    \[
        |f(t)^2 - g(t)^2| \leq (|f(t)| + |g(t)|)|f(t) - g(t)| \leq 2 c_2 u t^{2\alpha-2}\ell(t)^2 \|f-g\|_u, \qquad t \in I_u.
    \]
    Hence, evaluating $G_u$ yields
    \begin{align*}
        |G_u(f)(t) - G_u(g)(t)| 
        &\leq |\beta| \int_0^t \frac{(t-s)^{\alpha-1}}{\Gamma(\alpha)}\ell(t-s) s^{\alpha-1}\ell(s)\, \mathrm{d}s \| f-g\|_u
        \\ &\qquad + \frac{\sigma^2}{2} 2 c_2 u \int_0^t \frac{(t-s)^{\alpha-1}}{\Gamma(\alpha)}\ell(t-s) s^{2\alpha-2}\ell(s)^2\, \mathrm{d}s \| f-g\|_u
        \\ &\leq (1+\varepsilon) t^{\alpha-1}\ell(t) \| f - g\|_u \left[ |\beta| \frac{B(\alpha,\alpha)}{\Gamma(\alpha)}t^{\alpha}\ell(t) + \sigma^2 c_2 u A^{-1}_{\alpha} t^{2\alpha-1}\ell(t)^2 \right].
    \end{align*}
    Since $u t^{2\alpha-1}\ell(t)^2 \leq (1+\varepsilon)^2M$ for $t \in I_u$, and $t^\alpha\ell(t)$ converges to zero as $t \searrow 0$, we find $t_0 > 0$ small enough, and hence by \eqref{eq: t0 tu relation} $u_0$ large enough such that for $u \geq u_0$ and $t \in I_u$
    \begin{align*}
     &(1+\varepsilon) \left[ |\beta| \frac{B(\alpha,\alpha)}{\Gamma(\alpha)}t^{\alpha}\ell(t) + \sigma^2 c_2 u A^{-1}_{\alpha} t^{2\alpha-1}\ell(t)^2 \right] 
     \\ &\qquad \qquad < (1+\varepsilon) \left[ \varepsilon + (1+\varepsilon)^2\sigma^2 c_2 A^{-1}_{\alpha} M \right] < 1
    \end{align*}
    provided that $M$ satisfies
    \begin{align}\label{eq: M step 2}
        0 < M < \frac{A_{\alpha}}{\sigma^2 c_2 } \frac{1 - \varepsilon(1+\varepsilon)}{(1+\varepsilon)^{3}}. 
    \end{align}
    This shows that $G_u$ is a contraction and completes the proof of step 2.

    \textit{Step 3. Concluding the lower bound.} By Banach's fixed point theorem, there exists a unique solution $f_u \in \mathcal{U}_u$ of the equation $f_u = G_u(f_u)$. Since $f_u$ satisfies \eqref{eq: Laplace Volterra Riccati equation} by construction, uniqueness yields $f_u = V_u$ on $I_u$ with $u \geq u_0$. Since $V_{\infty}(t) = \lim_{u \to \infty} V_u(t)$, we evaluate along the parameterized trajectory determined by $u(t) = M t^{1-2\alpha} \ell(t)^{-2}$. For sufficiently small $t$, the trajectory $u(t)$ is sufficiently large, and hence we find $t \leq t_{u(t)}$ which gives $t \in I_{u(t)}$, and hence 
    \[
        V_{\infty}(t) \geq V_{u(t)}(t) = f_{u(t)}(t) \geq c_1 u(t) t^{\alpha-1}\ell(t) = c_1 M \frac{t^{-\alpha}}{\ell(t)}.
    \]
    To optimize the choices of the constants $c_1, c_2$, and $M$, we maximize the product $c_1 M$ subject to the invariance constraint from step~1, which by~\eqref{eq: 8}
    requires $c_1 \leq \frac{1}{\Gamma(\alpha)} - \varepsilon - \frac{\sigma^2}{2} c_2^2 A^{-1}_\alpha (1+\varepsilon)^3 M$. This yields the upper bound
    \[
        c_1 M \leq M \left( \frac{1}{\Gamma(\alpha)} - \varepsilon - (1+\varepsilon)^3\frac{\sigma^2}{2} c_2^2 A^{-1}_{\alpha} M \right).
    \]
    The right-hand side is a concave quadratic function of $M$, whose maximum is achieved at
    \[
        M^* = \frac{A_{\alpha}(1+\varepsilon)^{-3}}{\sigma^2 c_2^2} \left( \frac{1}{\Gamma(\alpha)} - \varepsilon \right).
    \]
    Remark that step 2 remains valid for this choice of $M^*$, since \eqref{eq: M step 2} reduces to $\frac{1}{\Gamma(\alpha)} - \varepsilon < c_2(1 - \varepsilon(1+\varepsilon))$ which is satisfied for sufficiently small $\varepsilon$. Evaluating the quadratic at $M^*$ yields 
    \[
        c_1 M^* = \frac{A_{\alpha}(1+\varepsilon)^{-3}}{2 \sigma^2 c_2^2} \left( \frac{1}{\Gamma(\alpha)} - \varepsilon \right)^2.
    \]
    By choosing $\varepsilon > 0$ arbitrarily small and letting $c_2 \searrow 1/\Gamma(\alpha)$, the product satisfies
    \[
        \lim_{\substack{\varepsilon \searrow 0 \\ c_2 \searrow 1/\Gamma(\alpha)}} c_1 M^* = \frac{A_{\alpha}}{2\sigma^2} = \frac{\Gamma(3\alpha-1)}{2\sigma^2 \Gamma(2\alpha-1)}.
    \]
    Therefore, for any given threshold $\eta \in (0,1)$, we can choose $c_2 > 1/\Gamma(\alpha)$ sufficiently close to its lower bound, $\varepsilon, t_0 > 0$ sufficiently small, and hence $u_0$ sufficiently large such that $c_1 M \geq (1-\eta)\frac{2A_\alpha}{\sigma^2}$. This completes the proof.
\end{proof}

\subsection{The regular case}

Below we study bounds for the unique solution $V_u$ of \eqref{eq: Laplace Volterra Riccati equation} in the case of regular Volterra kernels. The latter are crucial for the sufficiency on the boundary non-attainment and existence of negative moments.

\begin{Theorem}\label{Theorem: lower bound}
    Suppose that $K \in C^1(\R_+)$. Then the following assertions hold:
    \begin{enumerate}
        \item[(a)] The function $V_u$ satisfies the two-sided bounds
        \begin{equation}\label{eq:Vu geq}
            W_u(t) \leq V_u(t) \leq W_u(t) + K(0)L_0(0) \int_0^t W_u(s) \mathrm{e}^{\beta^+ K(0)(t-s)}\, \mathrm{d}s.
        \end{equation}
        where $\beta^+ = \max\{\beta, 0\}$. The function $W_u$ is for $\beta \neq L_0(0)$ given by
        \[
            W_u(t) = \frac{(\beta - L_0(0))K(0)u \mathrm{e}^{K(0)(\beta - L_0(0))t}}{\beta - L_0(0) + \frac{\sigma^2}{2}K(0)u \left( \mathrm{e}^{K(0)(\beta - L_0(0))t} - 1\right)} , \qquad t > 0,
        \]
        while for $\beta = L_0(0)$ it takes the form
        \[
            W_u(t) = \frac{K(0)u}{1 + \frac{\sigma^2 K(0)^2}{2}ut}, \qquad t > 0.
        \]
        \item[(b)] The function $R_u(t) := (L \ast V_u)(t)$ satisfies the upper bound 
        \begin{equation}\label{eq:Ru leq}
            R_u(t) \leq \frac{W_u(t)}{K(0)} + L_0(0)\int_0^t W_u(s) \mathrm{e}^{\beta^+ K(0)(t-s)}\, \mathrm{d}s.
        \end{equation}
        \item[(c)] For any $t>0$ such that $\beta < L_0(t)$, we have
        \[
            \int_0^t V_u(s)\, \mathrm{d}s \leq \frac{L_0(0) - \beta}{L_0(t) - \beta}\int_0^t W_u(s)\, \mathrm{d}s.
        \]
    \end{enumerate}
\end{Theorem}
\begin{proof}
    Let us first prove the lower bound. Since $K(0) < \infty$, using~\eqref{eq: L convolution V}
    combined with~\eqref{eq: resolvent first kind}, we find
    \[
        K(0)^{-1}V_u(t) + \int_0^t V_u(s)L_0(t-s)\, \mathrm{d}s = (L \ast V_u)(t) = u + \int_0^t \left( \beta V_u(s) - \frac{\sigma^2}{2} V_u(s)^2\right)\, \mathrm{d}s. 
    \]
    Since $|K'(0)| < \infty$, by Lemma \ref{lemma: first kind asymptotics} we find $L_0(0) = |K'(0)|/K(0)^2 \in \R_+$, and hence differentiating in $t$ and multiplying by $K(0)$ gives
    \begin{align*}
        V_u'(t) &= - K(0)L_0(0) V_u(t) - K(0) \int_0^t V_u(s)L_0'(t-s)\, \mathrm{d}s + K(0)\beta V_u(t) - \frac{\sigma^2 K(0)}{2}V_u(t)^2
        \\ &\geq K(0)\left( \beta - L_0(0)\right) V_u(t) - \frac{\sigma^2 K(0)}{2}V_u(t)^2,
    \end{align*}
    where we have used $L_0' \leq 0$ by complete monotonicity. Let $W_u$ be the unique solution of
    \begin{align}\label{eq: Wu equation}
        W_u'(t) = K(0)\left( \beta - L_0(0)\right) W_u(t) - \frac{\sigma^2K(0)}{2}W_u(t)^2, \qquad W_u(0) = K(0)u.
    \end{align}
    Since $W_u(0) = K(0)u = V_u(0)$, we find $W_u(t) \leq V_u(t)$. Solving \eqref{eq: Wu equation} proves the desired lower bound. 

    We proceed to prove the upper bound.  Using $L_0 \geq 0$, we find
    \[
        K(0)^{-1}V_u(t) \leq K(0)^{-1}V_u(t) + \int_0^t V_u(s)L_0(t-s)\, \mathrm{d}s = R_u(t).
    \]
    Firstly, assume that $\beta \leq 0$.
    From $0 \leq W_u(t) \leq V_u(t)$ we get $- V_u(t)^2 \leq - W_u(t)^2$, and hence the bound
    \[
        R_u'(t) = \beta V_u(t) - \frac{\sigma^2}{2}V_u(t)^2 \leq \beta W_u(t) - \frac{\sigma^2}{2}W_u(t)^2.
    \]
    Noting that $R_u(0) = u$ and using \eqref{eq: Wu equation}, we arrive at
    \begin{align*}
    R_u(t) &\leq u + \int_0^t \left( \beta W_u(s) - \frac{\sigma^2}{2} W_u(s)^2 \right)\, \mathrm{d}s
    \\ &= u + \int_0^t \left( \beta W_u(s) + \frac{W_u'(s)}{K(0)} - (\beta - L_0(0))W_u(s)\right)\, \mathrm{d}s
    \\ &= \frac{W_u(t)}{K(0)} + L_0(0)\int_0^t W_u(s)\, \mathrm{d}s.
    \end{align*}
    The assertion now follows from $V_u(t) \leq K(0)R_u(t)$. Let us now assume that $\beta > 0$. Using again $-V_u(t)^2 \leq - W_u(t)^2$ and $V_u(t) \leq K(0)R_u(t)$, we find
    \[
    R_u'(t) = \beta V_u(t) - \frac{\sigma^2}{2}V_u(t)^2 \leq \beta K(0)R_u(t) - \frac{\sigma^2}{2}W_u(t)^2.
    \]
    From this we obtain
    \[
    \frac{\mathrm{d}}{\mathrm{d}t}\left( R_u(t) \mathrm{e}^{-\beta K(0)t} \right) \leq - \frac{\sigma^2}{2} W_u(t)^2 \mathrm{e}^{-\beta K(0)t},
    \]
    and hence after integration, using $R_u(0) = u$ and~\eqref{eq: Wu equation},
    \begin{align*}
        R_u(t) &\leq u \mathrm{e}^{\beta K(0)t} - \frac{\sigma^2}{2} \int_0^t W_u(s)^2 \mathrm{e}^{\beta K(0)(t-s)}\, \mathrm{d}s
        \\ &= u \mathrm{e}^{\beta K(0)t} + \int_0^t \left( \frac{W_u'(s)}{K(0)} - (\beta - L_0(0))W_u(s)\right) \mathrm{e}^{\beta K(0)(t-s)}\, \mathrm{d}s 
        \\ &= u \mathrm{e}^{\beta K(0)t} + \int_0^t \frac{\mathrm{d}}{\mathrm{d}s}\left( \frac{W_u(s)}{K(0)}\mathrm{e}^{\beta K(0)(t-s)}\right) \, \mathrm{d}s + L_0(0)\int_0^t W_u(s)\mathrm{e}^{\beta K(0)(t-s)}\, \mathrm{d}s
        \\ &= \frac{W_u(t)}{K(0)} + L_0(0)\int_0^t W_u(s)\mathrm{e}^{\beta K(0)(t-s)}\, \mathrm{d}s,
    \end{align*}
    where we have used 
    \begin{align*}
        \frac{\mathrm{d}}{\mathrm{d}s}\left( \frac{W_u(s)}{K(0)}\mathrm{e}^{\beta K(0)(t-s)}\right) = \frac{W_u'(s)}{K(0)}\mathrm{e}^{\beta K(0)(t-s)} - \beta W_u(s)\mathrm{e}^{\beta K(0)(t-s)}.
    \end{align*}
    This proves assertion (b), and using $V_u(t) \leq K(0)R_u(t)$ also yields the desired upper bound in assertion (a).

    To prove assertion (c), let us first note that, by complete monotonicity of~$L_0$,
    \begin{align*}
        K(0)^{-1}V_u(t) + L_0(t)\int_0^t V_u(s)\, \mathrm{d}s \leq R_u(t) 
        = u + \int_0^t \left( \beta V_u(s) - \frac{\sigma^2}{2}V_u(s)^2\right)\, \mathrm{d}s.
    \end{align*}
    Rearranging terms and using $0 \leq W_u(t) \leq V_u(t)$, we find
    \begin{align*}
        (L_0(t) - \beta)\int_0^t V_u(s)\, \mathrm{d}s 
        &\leq u - \frac{W_u(t)}{K(0)} - \frac{\sigma^2}{2}\int_0^t W_u(s)^2\, \mathrm{d}s
        \\ &= u - \frac{W_u(t)}{K(0)} + \int_0^t \left( \frac{W_u'(s)}{K(0)} - (\beta - L_0(0))W_u(s)\right)\, \mathrm{d}s
        \\ &= (L_0(0) - \beta) \int_0^t W_u(s)\, \mathrm{d}s.
    \end{align*}
    Since $\beta < L_0(t) < L_0(0)$, dividing by $L_0(t) - \beta$ proves the assertion.
\end{proof}

\section{Boundary behaviour}

\subsection{Case of finite time}

Below we start with the case of regular Volterra kernels.

\begin{proof}[Proof of Theorem \ref{thm: boundary regular case}]
Recall that by Theorem \ref{Theorem: lower bound}, we have the bounds $W_u(t) \leq V_u(t) \leq \widetilde{W}_u(t)$, where $\widetilde{W}_u(t) = W_u(t) + K(0)L_0(0) \int_0^t W_u(s) \mathrm{e}^{\beta^+ K(0)(t-s)}\, \mathrm{d}s$, and $W_u(t)$ can be expressed as
\begin{align}\label{eq: Wu representation}
    W_u(t) = \frac{2}{\sigma^2 K(0)} \frac{\mathrm{d}}{\mathrm{d}t}\log\left( 1 + u \theta(t) \right)
\end{align}
with $\theta(t) = \frac{\sigma^2 K(0)^2}{2} \frac{\mathrm{e}^{\lambda t}-1}{\lambda}$ when $\lambda = K(0)(\beta - L_0(0)) \neq 0$, and $\theta(t) = \frac{\sigma^2 K(0)^2}{2} t$ when $\lambda = 0$. Let us denote by
\begin{align*}
    \Phi_u(t) &= x_0 ( V_u \ast L)(t) + b \int_0^t V_u(s)\, \mathrm{d}s
    \\ &= \frac{x_0}{K(0)}V_u(t) + x_0 \int_0^t V_u(s)L_0(t-s)\, \mathrm{d}s + b \int_0^t V_u(s)\, \mathrm{d}s
\end{align*}
the exponent of the Laplace transform determined by the relation $\E[\mathrm{e}^{-uX_t}] = \exp(-\Phi_u(t))$; see~\eqref{eq: affine trafo}, \eqref{eq:def V_u}
and~\eqref{eq: L convolution V}.  Using \eqref{eq: resolvent first kind} and the complete monotonicity of $L_0$ which gives $L_0(t-s) \geq L_0(t)$, we obtain the lower bound 
\begin{align*}
    \Phi_u(t) &\geq \frac{x_0}{K(0)}W_u(t) + (x_0 L_0(t) + b)\int_0^t W_u(s)\, \mathrm{d}s 
    \\ &= \frac{x_0}{K(0)}W_u(t) + \gamma_-(t) \log\left(  1 + u \theta(t) \right),
\end{align*}
where $\gamma_-(t) = \frac{2(x_0L_0(t) + b)}{\sigma^2 K(0)}$ and we have used \eqref{eq: Wu representation}. To obtain an upper bound for $\Phi_u(t)$, first suppose that $\beta \geq L_0(t)$. Substituting the upper bound $V_u(t) \leq \widetilde{W}_u(t)$ and the upper bound~\eqref{eq:Ru leq} for $R_u$ yields 
\begin{align*}
    \Phi_u(t) &= x_0 R_u(t) + b \int_0^t V_u(s)\, \mathrm{d}s
    \\ &\leq x_0 \frac{W_u(t)}{K(0)} + x_0 L_0(0)\int_0^t W_u(s)\mathrm{e}^{\beta^+ K(0)(t-s)}\, \mathrm{d}s + b \int_0^t \widetilde{W}_u(s)\, \mathrm{d}s.
\end{align*}
For the last integral, we use the particular form of $\widetilde{W}_u$ combined with Fubini's Theorem to find
\begin{align*}
    \int_0^t \widetilde{W}_u(s)\, \mathrm{d}s &= \int_0^t W_u(s)\, \mathrm{d}s + K(0)L_0(0) \int_0^t \int_0^s W_u(r)\mathrm{e}^{\beta^+ K(0)(s-r)}\, \mathrm{d}r \mathrm{d}s
    \\ &= \int_0^t W_u(s)\, \mathrm{d}s + K(0)L_0(0)\int_0^t W_u(s) \frac{\mathrm{e}^{\beta^+ K(0)(t-s)} - 1}{\beta^+ K(0)} \, \mathrm{d}s.
\end{align*}
Hence we obtain
\begin{align*}
    \Phi_u(t) &\leq \frac{x_0}{K(0)}W_u(t) + \int_0^t W_u(s) \Lambda(t-s)\, \mathrm{d}s,
\end{align*}
where $\Lambda$ is defined as in Theorem \ref{thm: boundary regular case}. Noting that $\Lambda$ is increasing, we may bound $\Lambda(t-s) \leq \Lambda(t)$, and hence obtain
\[
    \Phi_u(t) \leq \frac{x_0}{K(0)}W_u(t) + \Lambda(t)\int_0^t W_u(s)\, \mathrm{d}s
    = \frac{x_0}{K(0)}W_u(t) + \gamma_+(t) \log(1 + u \theta(t))
\]
with $\gamma_+(t) = \frac{2\Lambda(t)}{\sigma^2 K(0)}$, where we have used \eqref{eq: Wu representation}. Consequently, the Laplace transform is bounded on both sides by
\[
    \frac{\exp\left(- \frac{x_0}{K(0)}W_u(t)\right)}{ \left( 1 + u \theta(t) \right)^{\gamma_+(t)}} \leq \E\big[ \mathrm{e}^{-uX_t}\big] \leq \frac{\exp\left(- \frac{x_0}{K(0)}W_u(t)\right)}{\left( 1 + u \theta(t) \right)^{\gamma_-(t)}}.
\]
Finally, if $\beta < L_0(t)$, then we may use Theorem \ref{Theorem: lower bound}~(b)
and~(c) to find
\begin{align*}
    \Phi_u(t) &\leq \frac{x_0}{K(0)}W_u(t) + x_0 L_0(0)\int_0^t W_u(s)\mathrm{e}^{\beta^+ K(0)(t-s)}\, \mathrm{d}s + b \frac{L_0(0) - \beta}{L_0(t) - \beta} \int_0^t W_u(s)\, \mathrm{d}s
    \\ &\leq \frac{x_0}{K(0)}W_u(t) + \left(x_0 L_0(0)\mathrm{e}^{\beta^+ K(0)t} + b \frac{L_0(0) - \beta}{L_0(t) - \beta}\right)\int_0^t W_u(s)\, \mathrm{d}s
    \\ &= \frac{x_0}{K(0)}W_u(t) + \frac{2\left(x_0 L_0(0)\mathrm{e}^{\beta^+ K(0)t} + b \frac{L_0(0) - \beta}{L_0(t) - \beta}\right)}{\sigma^2 K(0)}\log(1 + u \theta(t)).
\end{align*}
In this case, we obtain the desired lower bound
\[
\E\left[ \mathrm{e}^{-u X_t}\right] \geq \frac{\exp\left( - \frac{x_0}{K(0)}W_u(t)\right)}{(1 + u \theta(t))^{\gamma_+(t)}}.
\]
Using the representation $x^{-p} = \frac{1}{\Gamma(p)}\int_0^{\infty} \mathrm{e}^{-xu}u^{p-1}\, \mathrm{d}u$, we obtain from the substitution $w = u \theta(t)$ the upper bound 
\begin{align*}
    \E[X_t^{-p}] &= \int_0^{\infty} \E\left[ \mathrm{e}^{-u X_t} \right] \frac{u^{p-1}}{\Gamma(p)}\, \mathrm{d}u
    \\ &\leq \int_0^{\infty} \exp\left(- \frac{x_0}{K(0)}W_u(t)\right) \left( 1 + u \theta(t) \right)^{-\gamma_-(t)} \frac{u^{p-1}}{\Gamma(p)}\, \mathrm{d}u
    \\ &= \theta(t)^{-p} \int_0^{\infty} \exp\left( - C(t) \frac{w}{1 + w}\right) \left( 1 + w\right)^{- \gamma_-(t)} \frac{w^{p-1}}{\Gamma(p)}\, \mathrm{d}w.
\end{align*}
Analogously, we obtain the lower bound
\begin{align}
    \E[X_t^{-p}] &\geq \int_0^{\infty} \exp\left(- \frac{x_0}{K(0)}W_u(t)\right) \left( 1 + u \theta(t) \right)^{-\gamma_+(t)} \frac{u^{p-1}}{\Gamma(p)}\, \mathrm{d}u \notag
    \\ &=  \theta(t)^{-p} \int_0^{\infty} \exp\left( - C(t) \frac{w}{1 + w}\right) \left( 1 + w\right)^{- \gamma_+(t)} \frac{w^{p-1}}{\Gamma(p)}\, \mathrm{d}w. \label{eq:lb X^-p}
\end{align}
If $x_0 = 0$, then after substitution, both sides are given by the Beta function $B(p, \gamma_{\pm}(t)-p)/\Gamma(p) = \frac{\Gamma(\gamma_{\pm}(t) - p)}{\Gamma(\gamma_{\pm}(t))}$, which gives the asserted bounds.

Suppose that $x_0 > 0$. For the upper bound (using $\gamma_-(t)$), we split the integral at $w = 1$. For $w \in [0,1]$, we use $\frac{w}{1+w} \geq \frac{w}{2}$ and $(1+w)^{-\gamma_-(t)} \leq 1$. Extending the domain to infinity yields
\begin{align*}
    \int_0^1 \exp\left( - C(t) \frac{w}{1 + w}\right) \left( 1 + w\right)^{- \gamma_-(t)} \frac{w^{p-1}}{\Gamma(p)}\, \mathrm{d}w 
    &\leq \int_0^\infty \exp\left( - \frac{C(t)}{2} w\right) \frac{w^{p-1}}{\Gamma(p)}\, \mathrm{d}w 
    = \left( \frac{C(t)}{2} \right)^{-p}.
\end{align*}
For $w \in [1,\infty)$, we use $\frac{w}{1+w} \geq \frac{1}{2}$ to obtain:
\begin{align*}
    \int_1^\infty \exp\left( - C(t) \frac{w}{1 + w}\right) \left( 1 + w\right)^{- \gamma_-(t)} \frac{w^{p-1}}{\Gamma(p)}\, \mathrm{d}w 
    &\leq \exp\left(- \frac{C(t)}{2}\right) \int_0^\infty (1+w)^{-\gamma_-(t)} \frac{w^{p-1}}{\Gamma(p)}\, \mathrm{d}w 
    \\ &= \exp\left(- \frac{C(t)}{2}\right) \frac{\Gamma(\gamma_-(t) - p)}{\Gamma(\gamma_-(t))}.
\end{align*}
Summing these components provides the upper bound. 

For the lower bound (using $\gamma_+(t)$), we use the global inequality $\frac{w}{1+w} \leq 1$ over $[0, \infty)$, implying $\exp\left(-C(t)\frac{w}{1+w}\right) \geq \mathrm{e}^{-C(t)}$. Factoring this constant out leaves the standard Beta integral
\begin{align*}
    \int_0^\infty \exp\left( - C(t) \frac{w}{1 + w}\right) \left( 1 + w\right)^{- \gamma_+(t)} \frac{w^{p-1}}{\Gamma(p)}\, \mathrm{d}w 
    &\geq \mathrm{e}^{-C(t)} \int_0^\infty (1+w)^{-\gamma_+(t)} \frac{w^{p-1}}{\Gamma(p)}\, \mathrm{d}w 
    \\ &= \mathrm{e}^{-C(t)} \frac{\Gamma(\gamma_+(t) - p)}{\Gamma(\gamma_+(t))}.
\end{align*}
Combining this with~\eqref{eq:lb X^-p} proves the assertion. 
\end{proof}

Next we continue with the proof for the rough case.
 \begin{proof}[Proof of Theorem \ref{Theorem: atom rough}]
    The probability of an atom at zero is given as the limit $u \to \infty$ in the Laplace transform. Hence, using again \eqref{eq: affine trafo} combined with \eqref{eq: L convolution V}, we find that
    \begin{align} 
        \mathbb{P}[X_t = 0] = \lim_{u \to \infty} \mathbb{E}\left[ e^{-u X_t} \right] 
        = \exp\left( - x_0 \int_0^t L_0(t-s) V_\infty(s)\, \mathrm{d}s - b \int_0^t V_\infty(s)\, \mathrm{d}s \right). \label{eq: 1}
    \end{align}
    
    \textit{Step 1: Local and global lower bounds.} By Theorem \ref{Theorem: upper bound rough}(a), for the general regularly varying case, we have the local upper bound $V_{\infty}(s) \leq \frac{2C_{\alpha}}{\sigma^2} \frac{s^{-\alpha}}{\ell(s)}$ for $s \in (0,t_0)$. Let $\tilde{\varepsilon} \in (0,1)$ be arbitrary. Then we find $t_1 \in (0,t_0)$ such that 
    \[
        \int_0^t V_{\infty}(s)\, \mathrm{d}s \leq \frac{2C_{\alpha}}{\sigma^2}\int_0^t \frac{s^{-\alpha}}{\ell(s)}\, \mathrm{d}s \leq (1+\tilde{\varepsilon}) \frac{C_*(\alpha)}{\sigma^2} \frac{t^{1-\alpha}}{1-\alpha}\ell(t)^{-1}, \qquad t \in (0,t_1),
    \]
    where we have used $2C_{\alpha} = - 2\Gamma(1-\alpha)/\Gamma(1-2\alpha) = C_*(\alpha)$. For the other term, Lemma \ref{lemma: first kind asymptotics}(a) yields $t_2 \in (0,t_1)$ such that $L_0(t) \leq \frac{1+\tilde{\varepsilon}}{\Gamma(1-\alpha)} \frac{t^{-\alpha}}{\ell(t)}$ holds for all $t \in (0,t_2)$. Hence, by Lemma \ref{lemma: convolution regularly varying} (applicable since $\ell^{-1}$ is slowly varying) we find $t_3 \in (0,t_2)$ such that 
    \begin{align*}
        \int_0^t L_0(t-s)V_{\infty}(s)\, \mathrm{d}s &\leq \frac{1+\tilde{\varepsilon}}{\Gamma(1-\alpha)}\frac{2C_{\alpha}}{\sigma^2}\int_0^t \frac{(t-s)^{-\alpha}}{\ell(t-s)}\frac{s^{-\alpha}}{\ell(s)}\, \mathrm{d}s
        \\ &\leq \frac{(1+\tilde{\varepsilon})^2}{\Gamma(1-\alpha)}\frac{2C_{\alpha}}{\sigma^2}B(1-\alpha, 1-\alpha) \frac{t^{1 - 2\alpha}}{\ell(t)^2}
    \end{align*}
    holds for $t \in (0,t_3)$. Using $B(1-\alpha, 1-\alpha) = \Gamma(1-\alpha)^2/\Gamma(2-2\alpha)$, the constant simplifies to $\Gamma(1-\alpha)/\Gamma(2-2\alpha)$. Choosing $\tilde{\varepsilon}$ small enough such that $(1+\tilde{\varepsilon})^2 \leq 1+\varepsilon$, and substituting these bounds directly into \eqref{eq: 1}, we can factor out $C_*(\alpha)/\sigma^2$ which yields the local lower bound. 
    
    For the pure fractional kernel, Theorem \ref{Theorem: upper bound rough}(b) provides the explicit global upper bound 
    \[
        V_{\infty}(s) \leq \frac{2C_{\alpha}}{\sigma^2} s^{-\alpha} + \mathbbm{1}_{\beta > 0} \frac{2\beta}{\sigma^2}, \qquad \forall s > 0.
    \]
    Integrating this bound over $[0, t]$ yields
    \begin{align*}
        \int_0^t V_{\infty}(s) \, \mathrm{d}s &\leq \int_0^t \left( \frac{2C_{\alpha}}{\sigma^2} s^{-\alpha} + \mathbbm{1}_{\beta > 0} \frac{2\beta}{\sigma^2} \right) \mathrm{d}s \\
        &= \frac{C_*(\alpha)}{\sigma^2(1-\alpha)} t^{1-\alpha} + \mathbbm{1}_{\beta > 0} \frac{2\beta}{\sigma^2} t.
    \end{align*}
    For the convolution integral, using $L_0(s) = \frac{s^{-\alpha}}{\Gamma(1-\alpha)}$, we have
    \begin{align*}
        \int_0^t L_0(t-s)V_{\infty}(s)\, \mathrm{d}s 
        &\leq \int_0^t \left(\frac{(t-s)^{-\alpha}}{\Gamma(1-\alpha)}\frac{2C_{\alpha}}{\sigma^2} s^{-\alpha} + \frac{s^{-\alpha}}{\Gamma(1-\alpha)}\mathbbm{1}_{\beta > 0} \frac{2\beta}{\sigma^2}\right)\, \mathrm{d}s
        \\ &= \frac{C_*(\alpha)}{\sigma^2} \frac{B(1-\alpha,1-\alpha)}{\Gamma(1-\alpha)} t^{1 - 2\alpha} + \mathbbm{1}_{\beta > 0}\frac{2\beta}{\sigma^2}\frac{t^{1 -\alpha}}{\Gamma(2-\alpha)}.
    \end{align*}
    Substituting these bounds into \eqref{eq: 1} proves the global lower bound. 

    \textit{Step 2: Local upper bounds.} Next, we establish the local upper bound. By Theorem \ref{Theorem: lower bound rough} we find for $\eta \in (0,1)$ some $t_0 > 0$ such that
    \[
        V_{\infty}(t) \geq (1-\eta) \frac{A_{\alpha}}{2\sigma^2}\frac{t^{-\alpha}}{\ell(t)} , \qquad t \in (0,t_0).
    \]
    Using this lower bound in \eqref{eq: 1} gives an upper bound on the probability:
    \begin{align*}
        \P[X_t = 0] &\leq \exp\left( - x_0(1-\eta) \frac{A_{\alpha}}{2\sigma^2}\int_0^t \frac{s^{-\alpha}}{\ell(s)}L_0(t-s)\, \mathrm{d}s - (1-\eta)b\frac{A_{\alpha}}{2\sigma^2} \int_0^t \frac{s^{-\alpha}}{\ell(s)}\, \mathrm{d}s \right).
    \end{align*}
    
    By adjusting $t_0$ smaller, say $0 < t_1 < t_0$, Lemma \ref{lemma: first kind asymptotics} provides $L_0(t) \geq (1-\eta) \frac{t^{-\alpha}}{\Gamma(1-\alpha)}\ell(t)^{-1}$, and we may use Lemma \ref{lemma: convolution regularly varying} and Karamatas Theorem for the second integral to bound 
    \begin{align*}
        &\   x_0(1-\eta) \frac{A_{\alpha}}{2\sigma^2}\int_0^t \frac{s^{-\alpha}}{\ell(s)}L_0(t-s)\, \mathrm{d}s + (1-\eta)b\frac{A_{\alpha}}{2\sigma^2} \int_0^t \frac{s^{-\alpha}}{\ell(s)}\, \mathrm{d}s
        \\ &\geq \frac{A_{\alpha}}{2\sigma^2} \left[ \frac{x_0 (1-\eta)^2 }{\Gamma(1-\alpha)} \int_0^t \frac{(t-s)^{-\alpha}s^{-\alpha}}{\ell(t-s)\ell(s)} \, \mathrm{d}s + b(1-\eta) \int_0^t \frac{s^{-\alpha}}{\ell(s)}\, \mathrm{d}s \right]
        \\ &\geq \frac{A_{\alpha}}{2\sigma^2}\left[ \frac{x_0(1-\eta)^3}{\Gamma(1-\alpha)}B(1-\alpha,1-\alpha)\frac{t^{1-2\alpha}}{\ell(t)^2} + \frac{b (1-\eta)^2}{1-\alpha} \frac{t^{1-\alpha}}{\ell(t)} \right].
    \end{align*}
    Using $\frac{B(1-\alpha,1-\alpha)}{\Gamma(1-\alpha)} = \frac{\Gamma(1-\alpha)}{\Gamma(2-2\alpha)}$ and noting that $\eta$ was arbitrary, proves the assertion.
\end{proof}

\subsection{Hitting time}

\begin{proof}[Proof of Theorem \ref{thm: boundary regular case}]
    Recall that $L$ has the form \eqref{eq: resolvent first kind}. Convolving \eqref{eq: VCIR} with $L$ yields 
    \[
        K(0)^{-1}X_t + \int_0^t L_0(t-s)X_s\, \mathrm{d}s = x_0 K(0)^{-1} + x_0 \int_0^t L_0(s)\, \mathrm{d}s + bt + \beta \int_0^t X_s\, \mathrm{d}s + \sigma \int_0^t \sqrt{X_s}\, \mathrm{d}B_s.
    \]
    Since $K \in C^1(\R_+)$, the process $X$ is a semimartingale. Multiplying by $K(0)$ gives the differential $\mathrm{d}X_t = D_t \, \mathrm{d}t + K(0)\sigma \sqrt{X_t}\, \mathrm{d}B_t$ with
    \[
    D_t = K(0)b + K(0)(\beta - L_0(0))X_t + x_0 K(0)L_0(t) - K(0)\int_0^t L_0'(t-s)X_s\, \mathrm{d}s.
    \]
    Since $K$ is completely monotone, also $L_0$ is completely monotone and hence $L_0 \geq 0$ and $L_0' \leq 0$. Fix $T \in (0,\infty]$ as in the assumptions. Since $X \geq 0$, we find $D_t \geq a + cX_t$ with
    \[
        a = K(0)b + x_0 K(0) L_0(T) \quad \text{ and } \quad c = K(0)(\beta - L_0(0)).
    \]
    An application of the comparison theorem \cite[Chapter VI, Theorem 1.1]{MR1011252}, yields the lower bound $X_t \geq Y_t$ on $[0,T]$, where $Y$ is the unique strong solution of
    \[
        \mathrm{d}Y_t = \left( a + cY_t\right)\, \mathrm{d}t + \sigma K(0) \sqrt{Y_t}\, \mathrm{d}B_t.
    \]
    Since $Y$ is a classical CIR process, its hitting time $\tau_0^Y$ of zero satisfies $\P[\tau_0^Y = \infty] = 1$ if and only if the Feller condition 
    \[
        1 \leq \frac{2( bK(0) + x_0K(0)L_0(T))}{\sigma^2 K(0)^2} = \frac{2(b + x_0 L_0(T))}{\sigma^2 K(0)}
    \]
    holds. Hence the assertion follows from
    \[
        \P\left[ \tau_0 \geq T \right] \geq \P[ \forall t \in (0,T): \ X_t > 0] \geq \P[ \forall t \in (0,T): \ Y_t > 0] = 1. \qedhere
    \]
\end{proof}

\subsection{Limit distribution}

Finally, we prove the boundary behaviour of the limit distribution.

\begin{proof}[Proof of Theorem \ref{thm: limit_distribution}]
    Using the affine transformation formula \eqref{eq: pi affine formula}, taking into account that $V_u(t) = - \psi(t;-u\delta_0)$, and finally noting~\eqref{eq: L convolution V}, we obtain
    \begin{align*}
        \int_{\R_+} \mathrm{e}^{-u x}\, \pi_{x_0}(\mathrm{d}x) 
        &= \exp\left( - x_0(L \ast V_u)(\infty) - b \int_0^{\infty}V_u(s)\, \mathrm{d}s \right)
        \\ &= \exp\left( - (x_0 L_0(\infty) + b)\int_0^{\infty} V_u(s)\, \mathrm{d}s\right),
    \end{align*}
    where the second equality follows from $V_u(t) \leq E_{\beta}(t)u \longrightarrow 0$ combined with
    \[
        (L \ast V_u)(\infty) = \lim_{t \to \infty}\left( K(0)^{-1}V_u(t) + \int_0^t L_0(t-s)V_u(s)\, \mathrm{d}s\right)
         = L_0(\infty)\int_0^{\infty}V_u(s)\, \mathrm{d}s.
    \]

    (a) Recall that Theorem \ref{Theorem: lower bound} provides the bound $V_u(t) \geq W_u(t)$. Hence we obtain the lower bound
    \begin{align*}
        \int_0^{\infty} V_u(t)\, \mathrm{d}t \geq \int_0^{\infty} W_u(t)\, \mathrm{d}t
        = \frac{2}{\sigma^2K(0)} \log(1 + u \theta(\infty)),
    \end{align*}
    where we have used \eqref{eq: Wu representation}, $\theta(0) = 0$, and since $\lambda = K(0)(\beta - L_0(0)) < 0$
    \[
        \theta(\infty) = \lim_{t \to \infty}\theta(t) = -\frac{\sigma^2 K(0)^2}{2\lambda} = \frac{\sigma^2 K(0)}{2(L_0(0) - \beta)}.
    \]
    
    Next, we derive an upper bound for $\int_0^{\infty} V_u(t)\, \mathrm{d}t$. Taking the limit $t \to \infty$ in \eqref{eq: L convolution V} and using $(L \ast V_u)(t) \longrightarrow L_0(\infty)\int_0^{\infty}V_u(t)\, \mathrm{d}t$, we obtain
    \[
        L_0(\infty) \int_0^{\infty}V_u(t)\, \mathrm{d}t = u + \int_0^{\infty} \left( \beta V_u(s) - \frac{\sigma^2}{2} V_u(s)^2 \right)\, \mathrm{d}s.
    \]
    and hence
    \begin{align*}
        (L_0(\infty) - \beta) \int_0^{\infty}V_u(t)\, \mathrm{d}t &= u - \frac{\sigma^2}{2} \int_0^\infty V_u(s)^2\, \mathrm{d}s
        \\ &\leq u - \frac{\sigma^2}{2} \int_0^\infty W_u(s)^2\, \mathrm{d}s
        \\ &= u + \int_0^{\infty} \left( \frac{W_u'(s)}{K(0)} - (\beta - L_0(0))W_u(s)\right)\, \mathrm{d}s
        \\ &= u + \frac{W_u(\infty) - W_u(0)}{K(0)} - (\beta - L_0(0))\int_0^{\infty} W_u(s)\, \mathrm{d}s
        \\ &= - \frac{2(\beta - L_0(0))}{\sigma^2 K(0)} \log\left( 1 + u \theta(\infty)\right)
    \end{align*}
    since $0 \leq W_u(t) \leq V_u(t)$ and hence $-V_u(t)^2 \leq - W_u(t)^2$, we have used $\frac{W_u'(t)}{K(0)} = (\beta - L_0(0)) W_u(t) - \frac{\sigma^2}{2} W_u(t)^2$ provided by \eqref{eq: Wu equation}, $W_u(\infty) = 0$, and $W_u(0) = K(0)u$. This proves the desired upper bound
    \[
        \int_0^{\infty}V_u(t)\, \mathrm{d}t \leq \frac{2}{\sigma^2 K(0)} \frac{L_0(0) - \beta}{L_0(\infty) - \beta} \log\left( 1 + u \theta(\infty) \right).
    \]
    Inserting the upper and lower bound for $\int_0^{\infty}V_u(t)\, \mathrm{d}t$ into the Laplace transform of the limit distribution gives 
    \[
        (1 + u \theta(\infty))^{-\gamma_-(\infty) \frac{L_0(0) - \beta}{L_0(\infty) - \beta}} \leq \int_{\mathbb{R}_+} \mathrm{e}^{-u x}\, \pi_{x_0}(\mathrm{d}x) \leq (1 + u \theta(\infty))^{-\gamma_-(\infty)},
    \]
    where $\gamma_-(\infty) = \frac{2(x_0 L_0(\infty) + b)}{\sigma^2 K(0)}$ and $\gamma_+(\infty) = \gamma_-(\infty) \frac{L_0(0) - \beta}{L_0(\infty) - \beta}$. Finally, using again the integral representation $x^{-p} = \frac{1}{\Gamma(p)} \int_0^\infty \mathrm{e}^{-ux} u^{p-1}\, \mathrm{d}u$ and Fubini's Theorem, gives the desired moment bounds. 

    (b) First we derive the lower bound on $\pi_{x_0}(\{0\})$. Note that $V_u(t) \nearrow V_{\infty}(t)$ for each $t > 0$ as $u \to \infty$. Let us first show that $V_{\infty} \in L^1(\R_+)$. Indeed, since for each $\eta \in (0,1)$ there exists $t_0 > 0$ such that $V_{\infty}(t) \leq (1+\eta)\frac{2C_{\alpha}}{\sigma^2} \frac{t^{-\alpha}}{\ell(t)}$ holds for $t \in (0,t_0)$ (see Theorem \ref{Theorem: upper bound rough}), it follows that $V_{\infty}$ is integrable on $[0,t_0)$. Using \eqref{eq: LVinfty}, we obtain for $t > t_0$
    \begin{align*}
        (-\beta)\int_{t_0}^t V_{\infty}(s)\, \mathrm{d}s = (L \ast V_\infty)(t_0) -  (L \ast V_\infty)(t) - \frac{\sigma^2}{2} \int_{t_0}^t V_\infty(s)^2 \, \mathrm{d}s
        \leq (L \ast V_{\infty})(t_0).
    \end{align*}
    Taking the limit $t \to \infty$ and noting that $\beta < 0$ yields
    \begin{align}\label{eq: PaulGassiat}
        \int_{t_0}^{\infty} V_{\infty}(t)\, \mathrm{d}t \leq (-\beta)^{-1}(L \ast V_{\infty})(t_0) < \infty
    \end{align}
    and hence proves $V_{\infty} \in L^1([t_0, \infty))$. Thus $V_{\infty} \in L^1(\R_+)$. By monotone convergence we find
    \begin{align*}
        \pi_{x_0}(\{0\}) 
        &= \lim_{u \to \infty} \int_{\R_+}\mathrm{e}^{-ux}\, \pi_{x_0}(\mathrm{d}x) 
        \\ &= \lim_{u \to \infty} \exp\left( - (x_0L_0(\infty) + b)\int_0^{\infty} V_u(t)\, \mathrm{d}t \right)
        \\ &= \exp\left( - (x_0 L_0(\infty) + b)\int_0^{\infty}V_{\infty}(t)\, \mathrm{d}t \right) > 0,
    \end{align*}
    which proves the first assertion. Likewise, we obtain from \eqref{eq: 1}
    \begin{align*}
        \P[X_t = 0] &= \lim_{u \to \infty} \E[\mathrm{e}^{-uX_t}]
        \\ &= \lim_{u \to \infty} \exp\left( - x_0 (L \ast V_{u})(t) - b \int_0^{t}V_u(s)\, \mathrm{d}s \right)
        \\ &= \exp\left( -x_0 (L \ast V_{\infty})(t) - b \int_0^t V_{\infty}(s)\, \mathrm{d}s\right) 
        \\ &\longrightarrow \exp\left( - (x_0 L_0(\infty) + b)\int_0^{\infty}V_{\infty}(s)\, \mathrm{d}s\right) = \pi_{x_0}(\{0\})
    \end{align*}
    as $t \to \infty$, where we have used $V_{\infty} \in L^1(\R_+)$ so that $(L \ast V_{\infty})(t) \longrightarrow L_0(\infty)\int_0^{\infty}V_{\infty}(t)\, \mathrm{d}t$.
\end{proof}

%%%%%%%%%%%%%%%%% Appendix

\appendix

\section{Auxiliary asymptotic results}

\subsection{Fractional Sobolev regularity}

\begin{Lemma}\label{lemma: fractional Sobolev}
    Let $K \in L_{\mathrm{loc}}^2(\R_+)$ satisfy \eqref{eq: K1} for some constant $\gamma \in (0,1)$. Then for each $\eta \in (0,\gamma)$, we find
    \[
        \int_0^T t^{-2\eta}|K(t)|^2\, \mathrm{d}t + \int_0^T \int_0^T \frac{|K(t) - K(s)|^2}{|t-s|^{1+2\eta}}\, \mathrm{d}s\mathrm{d}t < \infty.
    \]
\end{Lemma}
\begin{proof}
    We first bound the weighted $L^2$ norm. Let $F(t) = \int_0^t K(s)^2 \, \mathrm{d}s$, then $F(t) \leq C_T t^{2\gamma}$ by assumption for $t \in [0,T]$ and some constant $C_T > 0$. Using integration by parts, we obtain
    \[
        \int_0^T t^{-2\eta} |K(t)|^2 \, \mathrm{d}t = \int_0^T t^{-2\eta} \, \mathrm{d}F(t) = \left[ t^{-2\eta} F(t) \right]_0^T + 2\eta \int_0^T t^{-2\eta-1} F(t) \, \mathrm{d}t.
    \]
    Because $F(t) \leq C_T t^{2\gamma}$ and $\gamma > \eta$, the boundary term at zero vanishes. The integral term is bounded by
    \[
        2\eta \int_0^T t^{-2\eta-1} F(t) \, \mathrm{d}t \leq 2\eta C_T \int_0^T t^{2\gamma - 2\eta - 1} \, \mathrm{d}t = C_T \frac{\eta}{\gamma - \eta} T^{2\gamma - 2\eta} < \infty.
    \]
    This proves that the first integral is finite.

    For the second term, we use symmetry to restrict the integral onto $\{h < t\}$, and then the substitution $h = t - s$ to find
    \begin{align*}
        \int_0^T \int_0^T \frac{|K(t) - K(s)|^2}{|t-s|^{1+2\eta}} \, \mathrm{d}s \mathrm{d}t 
        &= 2 \int_0^T \int_0^t \frac{|K(t) - K(t-h)|^2}{h^{1+2\eta}} \, \mathrm{d}h \mathrm{d}t
        \\ &= 2 \int_0^T \frac{1}{h^{1+2\eta}} \left( \int_h^T |K(t) - K(t-h)|^2 \, \mathrm{d}t \right) \mathrm{d}h
        \\ &= 2 \int_0^T \frac{1}{h^{1+2\eta}} \left( \int_0^{T-h} |K(u+h) - K(u)|^2 \, \mathrm{d}u \right)\, \mathrm{d}h
        \\ &\lesssim \int_0^T h^{2\gamma - 1 - 2\eta}\, \mathrm{d}h.
    \end{align*}
    The last integral is finite since $\gamma > \eta$, which proves the assertion.
\end{proof}

\subsection{Convolutions of regularly varying functions}

\begin{Lemma}\label{lemma: convolution regularly varying}
    Let $\alpha, \beta \in (0,1)$ and let $f,g: (0,T] \longrightarrow \mathbb{R}$ be given by
    \[
        f(t) = t^{\alpha - 1}\ell_f(t) \quad \text{and} \quad g(t) = t^{\beta - 1}\ell_g(t), \qquad t > 0
    \]
    where $\ell_f, \ell_g: (0,T] \longrightarrow \mathbb{R}$ are measurable, and slowly varying at $t = 0$. Then as $t \searrow 0$,
    \[
        \int_0^t f(t-s)g(s)\, \mathrm{d}s \sim t^{\alpha+\beta-1} \ell_f(t)\ell_g(t) B(\alpha, \beta),
    \]
    where $B(\alpha, \beta) = \int_0^1 (1-y)^{\alpha-1}y^{\beta-1}\,\mathrm{d}y$ is the Beta function.
\end{Lemma}
\begin{proof}
    Since $\ell_f, \ell_g$ are slowly varying at zero, we may asume without loss of generality that they are stricly positive in a neighbourhood of zero. Using the substitution $s = ty$, so that $\mathrm{d}s = t\,\mathrm{d}y$, we find 
    \begin{align*}
        \int_0^t f(t-s)g(s)\, \mathrm{d}s 
        &= t^{\alpha+\beta-1} \ell_f(t)\ell_g(t) \int_0^1 (1-y)^{\alpha-1} y^{\beta-1} \frac{\ell_f(t(1-y))}{\ell_f(t)} \frac{\ell_g(ty)}{\ell_g(t)} \, \mathrm{d}y.
    \end{align*}
    Let $I_t(y)$ denote the integrand
    \[
        I_t(y) = (1-y)^{\alpha-1} y^{\beta-1} \frac{\ell_f(t(1-y))}{\ell_f(t)} \frac{\ell_g(ty)}{\ell_g(t)}.
    \]
    Since $\ell_f$ and $\ell_g$ are slowly varying at zero, we find $I_t(y) \longrightarrow (1-y)^{\alpha-1}y^{\beta-1}$ as $t \searrow 0$ with fixed $y \in (0,1)$. Potter's Theorem for slowly varying functions (see \cite[Theorem 1.5.6]{BGT87}) guarantees that for any $\delta > 0$, there exists a constant $A > 1$ and a threshold $t_0 > 0$ such that for all $x, y \in (0, t_0)$,
    \[
        \frac{\ell_f(x)}{\ell_f(y)} \leq A \max\left( \left(\frac{x}{y}\right)^\delta, \left(\frac{x}{y}\right)^{-\delta} \right)
    \]
    Take $0 < \delta < \min(\alpha, \beta)$ and let $t \in (0,t_0)$. Then we obtain for $y \in (0,1)$
    \[
            \frac{\ell_f(t(1-y))}{\ell_f(t)} \leq A (1-y)^{-\delta} \ \text{ and } \ \frac{\ell_g(ty)}{\ell_g(t)} \leq A y^{-\delta}.
    \]
    This provides the integrable upper bound $|I_t(y)| \leq A^2 (1-y)^{\alpha-1-\delta} y^{\beta-1-\delta}$ for $t \in (0,t_0)$ and $y \in (0,1)$. Thus, an application of dominated convergence, gives
    \[
        \lim_{t \searrow 0} \int_0^1 I_t(y) \, \mathrm{d}y = \int_0^1 (1-y)^{\alpha-1} y^{\beta-1} \, \mathrm{d}y = B(\alpha, \beta),
    \]
    and proves the assertion.
\end{proof}

In the next lemma we prove, under additional conditions, analogous asymptotics for the derivative of the convolution. For the definition and properties of smooth
variation, we refer to Section~1.8 in~\cite{BGT87}.

\begin{Lemma}\label{lemma: convolution derivative asymptotics}
    Let $\alpha, \beta \in (0,1)$ with $\alpha + \beta \neq 1$, and let $f,g: (0,T] \longrightarrow \mathbb{R}_+$ be given by
    \[
        f(t) = t^{\alpha - 1}\ell_f(t) \quad \text{and} \quad g(t) = t^{\beta - 1}\ell_g(t), \qquad t > 0,
    \]
    where $\ell_f, \ell_g: (0,T] \longrightarrow \mathbb{R}$ are measurable and smoothly varying at $t = 0$. Then as $t \searrow 0$,
    \[
    \frac{\mathrm{d}}{\mathrm{d}t} \int_0^t f(t-s)g(s)\, \mathrm{d}s \sim (\alpha+\beta-1) B(\alpha, \beta) t^{\alpha+\beta-2} \ell_f(t) \ell_g(t) \quad \text{as } t \searrow 0,
    \]
\end{Lemma}
\begin{proof}
    Firstly, since $\ell_f, \ell_g$ are smoothly varying, they satisfy by definition
    \[
        \lim_{t \searrow 0} \frac{t \ell_f'(t)}{\ell_f(t)} = 0 \quad \text{ and } \quad \lim_{t \searrow 0} \frac{t \ell_g'(t)}{\ell_g(t)} = 0.
    \]
    In particular, $f,g$ are smoothly varying and hence we find
    \begin{align}\label{eq: 5}
        \lim_{t \searrow 0}\frac{t f'(t)}{f(t)} = \alpha-1 \quad \text{ and } \quad \lim_{t \searrow 0}\frac{t g'(t)}{g(t)} = \beta - 1.
    \end{align}    
    Because $f$ and $g$ have a singularity in zero, to differentiate under the integral, we write $U(t) := \int_0^t f(t-s)g(s)\, \mathrm{d}s = \int_0^{t/2} f(t-s)g(s) \, \mathrm{d}s + \int_0^{t/2} f(s)g(t-s) \, \mathrm{d}s$, which yields for $t > 0$
\begin{align*}
    U'(t) &= \int_0^{t/2} f'(t-s)g(s) \, \mathrm{d}s + \int_0^{t/2} f(s)g'(t-s) \, \mathrm{d}s + f\left(\frac{t}{2}\right)g\left(\frac{t}{2}\right)
    \\ &= t \int_0^{1/2} f'(t(1-u)) g(tu) \, \mathrm{d}u + t \int_0^{1/2} f(tu) g'(t(1-u)) \, \mathrm{d}u 
    + f\left(\frac{t}{2}\right)g\left(\frac{t}{2}\right),
\end{align*}
where we have used the substitution $s = tu$ with $\mathrm{d}s = t \, \mathrm{d}u$. We divide $U'(t)$ by $t^{\alpha + \beta-2} \ell_f(t) \ell_g(t)$, and evaluate the pointwise limit of the integrands. For $I_1(t) = t \int_0^{1/2}f'(t(1-u))g(tu)\, \mathrm{d}u$, the normalized integrand satisfies for each $u \in (0,1/2)$
\begin{align*}
    H_{1,t}(u) &:= \frac{f'(t(1-u))}{t^{\alpha-2} \ell_f(t)} \frac{g(tu)}{t^{\beta-1} \ell_g(t)}
    \\ &= \frac{\ell_f(t(1-u))}{\ell_f(t)}\frac{tf'(t(1-u))}{f(t(1-u))} \frac{\ell_g(tu)}{\ell_g(t)} (1-u)^{\alpha-1} u^{\beta - 1}
    \longrightarrow (\alpha - 1)(1-u)^{\alpha - 2} u^{\beta - 1}
\end{align*}
since $\ell_f,\ell_g$ are slowly varying, and \eqref{eq: 5} gives $tf'(t(1-u))/f(t(1-u)) \longrightarrow (\alpha-1)/(1-u)$. To pass the limit inside the integral, by Potter's Bounds, for any chosen $\varepsilon > 0$, there exist constants $A_f, A_g > 0$ and $t_0 > 0$ such that for all $t \in (0, t_0]$ and $u \in (0, 1/2]$:
\[
    \left| \frac{\ell_f(t(1-u))}{\ell_f(t)} \right| \le A_f (1-u)^{-\varepsilon}, \quad \left| \frac{\ell_g(tu)}{\ell_g(t)} \right| \le A_g u^{-\varepsilon}.
\]
Since $1-u \ge 1/2$ for $u \in [0, 1/2]$, there exist constants $K_1, K_2$ such that
\[
    \left| \frac{f'(t(1-u))}{t^{\alpha-2} \ell_f(t)} \right| \le K_1 (1-u)^{\alpha-2-\varepsilon} \le K_1 2^{2-\alpha+\varepsilon} \ \text{ and } \  \left| \frac{g(tu)}{t^{\beta-1} \ell_g(t)} \right| \le K_2 u^{\beta-1-\varepsilon}.
\]
Consequently, $|H_{1,t}(u)|$ is uniformly dominated by an integrable function $C u^{\beta-1-\varepsilon}$. Since $\beta > 0$, we can choose $\varepsilon < \beta$ so that $\beta-1-\varepsilon > -1$, which guarantees integrability on $[0, 1/2]$. By the Dominated Convergence Theorem:
\[
    \lim_{t \searrow 0} \frac{I_1(t)}{t^{\alpha + \beta-2} \ell_f(t) \ell_g(t)} = \int_0^{1/2} (\alpha-1)(1-u)^{\alpha-2} u^{\beta-1} \, \mathrm{d}u.
\]
Applying the same argument to $I_2(t) = t\int_0^{1/2}f(tu)g'(t(1-u))\, \mathrm{d}u$, and choosing $\varepsilon < \alpha$ so that $\alpha - 1 - \varepsilon > -1$, shows that
\[
    \lim_{t \searrow 0} \frac{I_2(t)}{t^{\alpha + \beta-2} \ell_f(t) \ell_g(t)} = \int_0^{1/2}(\beta-1)(1-u)^{\beta - 2}u^{\alpha-1}\, \mathrm{d}u.
\]
Finally, using again the slowly varying property of $\ell_f, \ell_g$, we find
\[
    \frac{f\left( \frac{t}{2}\right) g\left( \frac{t}{2}\right)}{t^{\alpha+\beta - 2}\ell_f(t) \ell_g(t)}
    = 2^{2-\alpha - \beta} \frac{\ell_f\left( \frac{t}{2}\right)}{\ell_f(t)}\frac{\ell_g\left(\frac{t}{2}\right)}{\ell_g(t)} \longrightarrow 2^{2 - \alpha - \beta}.
\]
Hence we obtain $U'(t) \sim C t^{\alpha+\beta-2} \ell_f(t) \ell_g(t)$ as $t \searrow 0$ with the constant $C$ given by
\begin{align*}
    C = \int_0^{1/2} (\alpha-1)(1-u)^{\alpha-2} u^{\beta-1} \, \mathrm{d}u + \int_0^{1/2}(\beta-1)(1-u)^{\beta - 2}u^{\alpha-1}\, \mathrm{d}u + 2^{2-\alpha - \beta}.
\end{align*}
One can verify that $C = (\alpha + \beta - 1)B(\alpha,\beta)$, which proves the assertion.
\end{proof}

\subsection{Asymptotics of the resolvent of the first kind}

The next lemma determines the asymptotics of $L$ in terms of the asymptotics of $K$. 

\begin{Lemma}\label{lemma: first kind asymptotics}
    Let $0 \neq K \in L^1_{\mathrm{loc}}(\R_+)$ is completely monotone. Let $L$ be the corresponding resolvent of the first kind. The following assertions hold.
    \begin{enumerate}
        \item[(a)] Suppose that $K$ is regularly varying of the form \eqref{eq: regularly varying} with $\alpha \in (0,1)$ and $\ell$ slowly varying at zero. Then 
        \[
            L_0(t) \sim \frac{1}{\Gamma(1-\alpha)}\frac{t^{-\alpha}}{\ell(t)}, \qquad t \searrow 0.
        \]
        \item[(b)] Suppose that $K(0^+) < \infty$ and $-K'(t) \sim t^{\alpha - 1}\ell(t)$ for $\alpha \in (0,1)$ and a slowly varying function~$\ell$. Then 
        \[
            L_0(t) \sim \frac{|K'(t)|}{K(0^+)^2}, \qquad t \searrow 0.
        \]
        \item[(c)] Suppose that $K(0^+) < \infty$ and $|K'(0^+)| < \infty$. Then 
        \[
            L_0(0^+) = \frac{|K'(0^+)|}{K(0^+)^2}.
        \]
        \item[(d)] One has
        \[
            L_0(\infty) = \left( \int_0^{\infty}K(t)\, \mathrm{d}t \right)^{-1}.
        \]
    \end{enumerate}
\end{Lemma}
\begin{proof}
    (a) Karamata's Tauberian Theorem \cite[Theorem 1.7.6]{BGT87} applied to $K(t) = \frac{t^{\alpha-1}}{\Gamma(\alpha)}\ell(t)$ gives for its Laplace transform the asymptotics $\widehat{K}(z) \sim z^{-\alpha}\ell(1/z)$ as $z \to \infty$. Hence, because $K(0^+) = \infty$, we have $\widehat{L}(z) = \widehat{L}_0(z)$ and the resolvent identity yields
    \[
        \widehat{L}_0(z) = \frac{1}{z \widehat{K}(z)} \sim \frac{z^{\alpha-1}}{\ell(1/z)}, \qquad z \to \infty.
    \]
    Set $f(t) = \int_0^t L_0(s)\, \mathrm{d}s$. Then $\widehat{f}(z) = \widehat{L}_0(z)/z \sim z^{\alpha-2}/ \ell(1/z)$. Karamata's Tauberian Theorem yields
    \[
        f(t) \sim \frac{1}{\Gamma(2-\alpha)}\frac{t^{1-\alpha}}{\ell(t)}, \qquad t \searrow 0.
    \]
    Since $L_0$ is completely monotone, it is strictly decreasing. We may apply the Monotone Density Theorem (differentiating the asymptotic equivalence) and use $(1-\alpha)/\Gamma(2-\alpha) = 1/\Gamma(1-\alpha)$ to conclude the assertion.

    (b) Since $K(0^+) < \infty$ and $K$ is not identically zero, we find $K(0^+) = K(0) > 0$. The definition of $L$ yields $\widehat{L}(z) = \frac{1}{K(0)} + \widehat{L}_0(z)$. By the Volterra resolvent identity, $\widehat{L}(z) = \frac{1}{z\widehat{K}(z)}$. Using integration by parts on the Laplace transform of $K$, we obtain $z\widehat{K}(z) = K(0) + \widehat{K'}(z)$. Equating these expressions yields
    \[
        \frac{1}{K(0)} + \widehat{L}_0(z) = \frac{1}{K(0) + \widehat{K'}(z)} = \frac{1}{K(0)} \left( 1 + \frac{\widehat{K'}(z)}{K(0)} \right)^{-1}.
    \]
    Because $K$ is completely monotone, $-K' \geq 0$ and $\widehat{K'}(z) \to 0$ as $z \to \infty$. Expanding the right-hand side gives
    \[
        \widehat{L}_0(z) \sim - \frac{\widehat{K'}(z)}{K(0)^2}, \qquad z \to \infty.
    \]
    Since both $L_0$ and $-K'$ are completely monotone, and $-K'$ is regularly varying, an application of \cite[Theorem 1.7.6]{BGT87} shows that $L_0(t) \sim \frac{-K'(t)}{K(0)^2} = \frac{|K'(t)|}{K(0)^2}$ as $t \searrow 0$. 

    (c) Finally, if $K(0^+) < \infty$ and $|K'(0^+)| < \infty$, applying the Initial Value Theorem yields
    \[
        L_0(0^+) = \lim_{z \to \infty} z \widehat{L}_0(z) = \lim_{z \to \infty} \left( - \frac{z\widehat{K'}(z)}{K(0)^2} \right) = - \frac{K'(0^+)}{K(0)^2} = \frac{|K'(0^+)|}{K(0^+)^2}. 
    \]

(d) Recall from the definition of the resolvent of the first kind that $\widehat{L}(z) = \frac{1}{z \widehat{K}(z)}$. Since $L(\mathrm{d}t) = K(0)^{-1}\delta_0(\mathrm{d}t) + L_0(t)\mathrm{d}t$, its Laplace transform is given by $\widehat{L}(z) = K(0)^{-1} + \widehat{L}_0(z)$. Hence we obtain
    \[
        z\widehat{L}_0(z) = \frac{1}{\widehat{K}(z)} - z K(0)^{-1}.
    \]
    Since $L_0$ is completely monotone, it is nonincreasing and hence an application of the Final Value Theorem for Laplace transforms yields $L_0(\infty) = \lim_{z \searrow 0} z\widehat{L}_0(z)$. Because $K(t) \geq 0$, the Monotone Convergence Theorem implies that $\lim_{z \searrow 0} \widehat{K}(z) = \int_0^{\infty} K(t)\,\mathrm{d}t \in (0, \infty]$. Hence we obtain with the usual convention $\infty^{-1} = 0$ the desired relation.
\end{proof}

\section{Comparison principles for generalised Riemann-Liouville equations}\label{se:comp}

Let $L \in L^1(0, T) \cap C^1(0, T]$ be a kernel satisfying 
\[
    L(t) > 0 \ \text{ and } \ L'(t) \leq 0 \ \ \text{ for all } t \in (0,T] \ \text{ and } \ \lim_{t \searrow 0}tL(t) = 0,
\]
and
\begin{equation}\label{eq: appendix 1}
    \lim_{h \to 0^+} \frac{1}{h} \int_0^h L(u) u^\gamma \, \mathrm{d}u = 0.
\end{equation}
The standard fractional kernel $L(t) = \frac{t^{-\alpha}}{\Gamma(1-\alpha)}$ for $\alpha \in (0, 1)$ is one example associated with the fractional Riemann-Liouville kernel. More generally, any completely monotone $L$ this is regularly varying at zero with index $\alpha - 1 < 0$ satisfies these conditions. 

As a first step, we prove an extremum principle for functions in $S_{L,\gamma}$ as introduced below.

\begin{Definition}
    For $\gamma \in (0,1]$, the function space $S_{L, \gamma}$ consists of all functions $f \in C_{\mathrm{loc}}^{\gamma}(0, T] \cap L^1(0, T)$ such that $L \ast f \in C^1(0,T])$. For $f \in S_{L,\gamma}$, the generalized fractional derivative is defined for $t \in (0, T]$ as:
    \begin{equation}
        (D_L f)(t) = \frac{\mathrm{d}}{\mathrm{d}t} \int_0^t L(t-s)f(s)\, \mathrm{d}s
    \end{equation}
\end{Definition}
For $L(t) = t^{-\alpha}/\Gamma(1-\alpha)$ with $\alpha \in (\frac{1}{2},1)$, condition \eqref{eq: appendix 1} requires that $\gamma > \alpha$. 

\begin{Theorem}[Extremum Principle for rough functions] \label{thm:extremum}
Let $f \in S_{L, \gamma}$. If $f(t) \geq 0$ for all $t \in (0, t_*]$ and $f(t_*) = 0$ for some $t_* \in (0, T)$, then
\begin{equation}
    (D_L f)(t_*) \le 0.
\end{equation}
\end{Theorem}
\begin{proof}
Let $D(t) = \int_0^t L(t-s)f(s)\, \mathrm{d}s$. Because $D \in C^1(0, T]$ by assumption, its derivative at $t_*$ exists and coincides with its left-hand derivative
\begin{equation}
    (D_L f)(t_*) = D'(t_*) = \lim_{h \to 0^+} \frac{D(t_*) - D(t_*-h)}{h}.
\end{equation}
Splitting the integrals, we obtain for the right-hand side
\begin{align*}
    D(t_*) - D(t_*-h) &= \int_0^{t_*-h} \big[ L(t_*-s) - L(t_*-h-s) \big] f(s)\, \mathrm{d}s + \int_{t_*-h}^{t_*} L(t_*-s)f(s)\, \mathrm{d}s
    \\ &\leq \int_{t_*-h}^{t_*} L(t_*-s)f(s)\, \mathrm{d}s,
\end{align*}
where we have used $t_*-s > t_*-h-s$ so that $L(t_*-s) - L(t_*-h-s) \le 0$ since $L$ is nonincreasing, and that $f(s) \ge 0$ for $s \in (0, t_*]$. Changing variables via $u = t_* - s$, we obtain:
\begin{equation*}
    \frac{D(t_*) - D(t_*-h)}{h} \le \frac{1}{h} \int_0^h L(u)f(t_*-u)\, \mathrm{d}u
\end{equation*}
Because $f(t_*) = 0$ and $f \in C_{\mathrm{loc}}^{\gamma}((0,T])$, there exists a constant $C > 0$ such that $|f(t_*-u)| \le C u^\gamma$ for small $u$. Substituting this bound gives
\begin{equation*}
    \lim_{h \searrow 0}\frac{D(t_*) - D(t_*-h)}{h} \le \lim_{h \searrow 0}\frac{C}{h} \int_0^h L(u)u^\gamma\, \mathrm{d}u = 0. \qedhere
\end{equation*}
\end{proof}

Next we continue with comparison principles for fractional equations with nonlinear right-hand sides. We start with the case of strict inequalities. Remark that in the formulation below, due to the strict inequalities, no further assumptions on the nonlinearity are required.

\begin{Theorem}[Strict Comparison Principle] \label{thm:strict}
Let $f_1, f_2 \in S_{L, \gamma}$, and let $R: [0,T] \times \R \longrightarrow \R$ be measurable. Suppose that for $t \in (0,T]$
\begin{align}\label{eq: appendix 2}
    (D_L f_1)(t) + R(t, f_1(t)) < (D_L f_2)(t) + R(t, f_2(t)), 
\end{align}
and the initial conditions satisfy $\liminf_{t \to 0^+} \big( f_2(t) - f_1(t) \big) > 0$. Then 
\[
    f_1(t) < f_2(t), \qquad t \in (0, T].
\]
\end{Theorem}
\begin{proof}
Let $f(t) = f_2(t) - f_1(t)$. The asymptotic condition guarantees there is some $\delta > 0$ such that $f(t) > 0$ on $(0, \delta)$. Assume by contradiction that the assertion is not true. Since $f_1, f_2$ are continuous on $(0,T]$, there exists a first crossing point $t_* \in [\delta, T]$ such that
\begin{equation*}
    f(t) > 0 \text{ for } t \in (0, t_*), \quad \text{and} \quad f(t_*) = 0.
\end{equation*}
Applying the Extremum Principle (Lemma \ref{thm:extremum}) to $f(t)$ at $t_*$ gives
$(D_L f)(t_*) \le 0$, which by linearity of the generalised fractional derivative gives $(D_L f_2)(t_*) \le (D_L f_1)(t_*)$. 
Since $f_1(t_*) = f_2(t_*)$, we have $R(t_*, f_1(t_*)) = R(t_*, f_2(t_*))$. Hence we obtain 
\begin{equation*}
    (D_L f_2)(t_*) + R(t_*, f_2(t_*)) 
    \le (D_L f_1)(t_*) + R(t_*, f_1(t_*)),
\end{equation*}
which contradicts assumption \eqref{eq: appendix 2}. Thus, no crossing point can exist, and $f_1(t) < f_2(t)$ holds globally.
\end{proof}

Finally, we relax the strict differential inequality. Because $R(t, y)$ is assumed to be nondecreasing, we do not require any Lipschitz conditions to prove the extension.

\begin{Theorem}[Nonstrict Comparison Principle] \label{thm:nonstrict}
Let $f_1, f_2 \in S_{L, \gamma}$, and $R: [0,T] \times \R \longrightarrow \R$ be jointly measurable and nondecreasing in $y \in \R$ with $t \in (0,T]$ fixed. Assume that
\[
    (D_L f_1)(t) + R(t, f_1(t)) \le (D_L f_2)(t) + R(t, f_2(t)), \qquad t \in (0,T]
\]
holds, and the initial conditions satisfy $\liminf_{t \to 0^+} \big( f_2(t) - f_1(t) \big) \ge 0$. Then
\[
    f_1(t) \le f_2(t), \qquad t \in (0,T].
\]
\end{Theorem}
\begin{proof}
Let $\varepsilon > 0$. Define the supersolution $f^{\varepsilon}_2(t) = f_2(t) + \varepsilon$, which again belongs to $S_{L,\gamma}$. Its fractional derivative satisfies $(D_L f^{\varepsilon}_2)(t) = (D_L f_2)(t) + \varepsilon L(t)$. Because $R$ is nondecreasing and $f^{\varepsilon}_2(t) > f_2(t)$, we have $R(t, f^{\varepsilon}_2(t)) \geq R(t, f_2(t))$. Evaluating the operator on $f^{\varepsilon}_2$ gives
\begin{align*}
    (D_L f^{\varepsilon}_2)(t) + R(t, f^{\varepsilon}_2(t)) 
    &= (D_L f_2)(t) + \varepsilon L(t) + R(t, f^{\varepsilon}_2(t)) 
    \\ &\ge (D_L f_2)(t) + \varepsilon L(t) + R(t, f_2(t)) 
    \\ &\ge (D_L f_1)(t) + \varepsilon L(t) + R(t, f_1(t)) 
    \\ &> (D_L f_1)(t) + R(t, f_1(t)),
\end{align*}
where the last inequality follows from $\varepsilon L > 0$ by assumption. Hence $f^{\varepsilon}_2$ satisfies the strict differential inequality, and $\liminf_{t \to 0^+} (f^{\varepsilon}_2(t) - f_1(t)) \ge \varepsilon > 0$. By Theorem \ref{thm:strict}, $f_1(t) < f_2(t) + \varepsilon$ for all $t \in (0,T]$. Taking the limit as $\varepsilon \to 0^+$ gives $f_1(t) \le f_2(t)$.
\end{proof}

One cannot omit the assumption that $y \mapsto R(t,y)$ is nondecreasing. To see this, consider the half-order Riemann-Liouville fractional derivative with kernel $L(t) = \frac{t^{-1/2}}{\Gamma(1/2)} = \frac{1}{\sqrt{\pi t}}$. Define the strictly \textit{decreasing} nonlinearity $R(t,y) = - \frac{2}{\sqrt{\pi}} \sqrt{\max\{0,y\}}$. Let $f_1(t) = t$ and $f_2(t) = 0$, which both belong to $S_{L, 1}$ and satisfy the identical initial condition $\liminf_{t \to 0^+} (f_2(t) - f_1(t)) = 0$. Evaluating the generalized fractional derivative for $f_1$, we find 
\[
        (D_L f_1)(t) = \frac{\mathrm{d}}{\mathrm{d}t} \int_0^t \frac{1}{\sqrt{\pi(t-s)}} s \, \mathrm{d}s = \frac{\mathrm{d}}{\mathrm{d}t} \left( \frac{4}{3\sqrt{\pi}} t^{3/2} \right) = \frac{2}{\sqrt{\pi}} t^{1/2}.
\]
Consequently, by definition of $R$, we find $(D_L f_1)(t) + R(t, f_1(t)) = 0$. For the trivial solution $f_2 \equiv 0$, we clearly have $(D_L f_2)(t) + R(t, f_2(t)) = 0$. Thus, the differential inequality 
\[
        (D_L f_1)(t) + R(t, f_1(t)) \le (D_L f_2)(t) + R(t, f_2(t))
\]
is satisfied (with equality) for all $t \in (0,T]$. All assumptions of Theorem \ref{thm:nonstrict} are satisfied except for the monotonicity of $R$. However, the comparison principle fails since $f_1(t) > f_2(t)$ on $(0,T]$.

\section{Characterisation of subcriticality}

\begin{Lemma}\label{lemma: subcriticality}
    Let $K \in L_{\mathrm{loc}}^2(\R_+)$ be completely monotone, not identically zero, and let $L$ be its resolvent of the first kind given by \eqref{eq: resolvent first kind}. For $\beta \in \R$, let $0 \leq E_{\beta} \in L_{\mathrm{loc}}^2(\R_+)$ be given by \eqref{eq: Ebeta definition}. Then $E_{\beta} \in L^1(\R_+)$ holds if and only if $\beta < L_0(\infty)$. In such a case we find 
    \[
            \int_0^{\infty}E_{\beta}(t)\, \mathrm{d}t \leq \frac{1}{L_0(\infty) - \beta}.
    \]
\end{Lemma}
\begin{proof} 
 Convolving the equation for $E_{\beta}$ with $L$ gives
 \[
    K(0)^{-1}E_{\beta}(t) + \int_0^t L_0(t-s)E_{\beta}(s)\, \mathrm{d}s = (L \ast E_{\beta})(t) = 1 + \beta \int_0^t E_{\beta}(s)\, \mathrm{d}s.
 \]
 Set $H(t) = L_0(t) - \beta$. Then $E_{\beta}$ satisfies
 \begin{align}\label{eq: 6}
    K(0)^{-1}E_{\beta}(t) + \int_0^t H(t-s)E_{\beta}(s)\, \mathrm{d}s = 1.
 \end{align}

 Suppose that $\beta < L_0(\infty)$. Let us first show that $E_{\beta} > 0$ is completely monotone. By assumption, $H(t) \geq L_0(\infty) - \beta > 0$ is completely monotone. If $K(0) = \infty$, then $\int_0^t H(t-s)E_{\beta}(s)\, \mathrm{d}s = 1$, and hence $E_{\beta}$ is the function resolvent of the first kind for $H$. Since $H$ is completley monotone, also $E_{\beta}$ must be completely monotone and hence nonnegative. Since it is not constant (as $K$ is not constant), it is strictly positive. Assume now that $K(0) < \infty$. Let $R_{\beta} \in L_{\mathrm{loc}}^2(\R_+)$ be the resolvent of the second kind for $K(0)H$, i.e. the unique solution of $R_{\beta}(t) + K(0)\int_0^t H(t-s)R_{\beta}(s)\, \mathrm{d}s = K(0)H(t)$. Then $R_{\beta}$ is completely monotone, and according to the solution theory for linear Volterra equations, $E_{\beta}$ is given by
 \[
    E_{\beta}(t) = K(0)\left( 1 - \int_0^t R_{\beta}(s)\, \mathrm{d}s\right).
 \]
 Using this formula, we can prove that $E_{\beta}$ is completely monotone and hence strictly positive. 

 Next we prove the asserted integrability for $E_{\beta}$. Since $E_{\beta}$ is strictly positive, we obtain from \eqref{eq: 6}
 \[
    (L_0(\infty) - \beta)\int_0^t E_{\beta}(s)\, \mathrm{d}s \leq \int_0^t H(t-s)E_{\beta}(s)\, \mathrm{d}s = 1 - K(0)^{-1}E_{\beta}(t) \leq 1.
 \]
 This proves that 
 \[
    \int_0^{\infty} E_{\beta}(t)\, \mathrm{d}t \leq \frac{1}{L_0(\infty) - \beta}.
 \]
 Since $E_{\beta}$ is completely monotone, it satisies $E_{\beta}(t) \leq E_{\beta}(1)$ for $t \geq 1$, and hence 
 \[
    \int_{0}^{\infty}E_{\beta}(t)^2\, \mathrm{d}t \leq \int_0^1 E_{\beta}(t)^2\, \mathrm{d}t + E_{\beta}(1)\int_1^{\infty}E_{\beta}(t)\, \mathrm{d}t < \infty.
 \]

 Finally, suppose that $\beta \geq L_0(\infty)$. Assume by contradiction that $E_{\beta}$ would be integrable. Set $M = \int_0^{\infty} E_{\beta}(t)\, \mathrm{d}t \in (0,\infty)$, and recall that $H(t) = L_0(t) - \beta \searrow L_0(\infty) - \beta$. If $K(0) = \infty$, taking the limit $t \to \infty$ in \eqref{eq: 6} shows that $(L_0(\infty) - \beta)M = 1$, which is a contradiction to $L_0(\infty) - \beta \leq 0$. Hence $E_{\beta} \not \in L^1(\R_+)$. If $K(0) < \infty$, then isolating first $E_{\beta}(t)$ in \eqref{eq: 6} and then taking the limit $t \to \infty$, shows that the limit $E_{\beta}(t) \longrightarrow E_{\beta}(\infty)$ exists and is given by 
 \[
    E_{\beta}(\infty) = K(0) - K(0)(L_0(\infty) - \beta) = K(0)(\beta + 1 - L_0(\infty)) \geq K(0) > 0.
 \]
 Hence $E_{\beta}$ cannot be integrable, which proves the assertion.
\end{proof}

\section*{Acknowledgement}
Martin Friesen would like to thank Peng Jin for inspiring discussions on the boundary attainment problem for the rough square-root process. The authors would like to thank Paul Gassiat for pointing out an error in Theorem~\ref{Theorem: lower bound} in the first version of the manuscript. Inequality~\eqref{eq: PaulGassiat} was proposed by Paul Gassiat for the fractional case.

%%%%%%%%% Literature

\bibliographystyle{plainurl}
\bibliography{literature}

\end{document}